\documentclass[a4paper]{amsart}
\usepackage[margin=2cm]{geometry}
\usepackage[english]{babel}
\usepackage{amssymb,amsmath,amsthm,amsfonts}
\usepackage[colorlinks]{hyperref}
\usepackage{dsfont}
\usepackage[showonlyrefs]{mathtools}
\mathtoolsset{showonlyrefs=true}

\newcommand{\N}{\mathbb{N}}

\newcommand{\R}{\mathbb{R}}

\newcommand{\eps}{\varepsilon}
\newcommand{\Om}{\Omega}

\newcommand{\G}{\mathcal G_{\eps,\eta}}
\newcommand{\cp}{\mathrm{cap}}

\newcommand{\uu}{\widetilde{u}}
\newcommand{\vr}{v_\rho}
\newcommand{\la}{\lambda}
\def\Omt{\widetilde{\Om}}
\def\Omh{\widehat{\Om}}
\def\mh{\widehat{m}}

\def\N{\mathbb{N}}

\def\rbb{\mathbb{R}}
\def\R{\mathbb{R}}
\def\rc{\mathcal{R}}

\def\ut{\widetilde{u}}

\def\de{\delta}

\def\hc{\mathcal{H}}
\def\H{\mathcal{H}}

\def\bal{\begin{aligned}}
\def\eal{\end{aligned}}
\def\case#1#2{\par\noindent{\underline{\it Case~#1.}}\emph{ #2}\\}

\newtheorem{proposition}{Proposition}[section]
\newtheorem{theorem}[proposition]{Theorem}
\newtheorem{corollary}[proposition]{Corollary}
\newtheorem{lemma}[proposition]{Lemma}

\newtheorem{definition}[proposition]{Definition}
\newtheorem{remark}[proposition]{Remark}
\newcommand{\beq}{\begin{equation}}
\newcommand{\eeq}{\end{equation}}
\newcommand{\ben}{\begin{enumerate}}
\newcommand{\een}{\end{enumerate}}
\newcommand{\bit}{\begin{itemize}}
\newcommand{\eit}{\end{itemize}}

\newcommand{\mean}[1]{\,-\hskip-1.08em\int_{#1}} %media integrale displayed
\title{A spectral shape optimization problem with a nonlocal competing term}
\author{Dario Mazzoleni}
\address{Dipartimento di Matematica F. Casorati\\
Universit\`a di Pavia\\
Via Ferrata 5, 27100 Pavia (Italy)}
\email{dario.mazzoleni@unipv.it}

\author{Berardo Ruffini}
\address{Dipartimento di Matematica\\
Universit\`a di Bologna\\
Piazza di Porta San Donato 5, 40126 Bologna (Italy) }
\email{berardo.ruffini@unibo.it}



\thanks{This paper has been partially supported by the INdAM-GNAMPA project 2019 ``Ottimizzazione spettrale non lineare''. B. Ruffini was also supported by the project ANR-18-CE40-0013 SHAPO financed by the
French Agence Nationale de la Recherche (ANR)}

\begin{document}
	
\begin{abstract}
We study the minimization of a spectral functional made as the sum of the first eigenvalue of the Dirichlet Laplacian and the relative strength of a Riesz-type interaction functional. We show that when 
the Riesz repulsion strength is below a critical value, existence of minimizers occurs.  Then we prove, by means of an expansion analysis, that the ball is a rigid minimizer when the Riesz repulsion is small enough.
 { Eventually we show that for certain regimes of the Riesz repulsion, regular minimizers do not exist.}
\end{abstract}
	
\maketitle

\tableofcontents

%%%%%%%%%%%%%%%%%%%%%
\section{Introduction}
\subsection*{Foreword}
In this work we study the minimization under volume constraint of energies of the form
\[
\mathcal F(\Omega)= \mathcal S(\Omega)+\int_\Omega\int_\Omega\frac{dx\,dy}{|x-y|^{N-\alpha}},
\]
where $\mathcal S$ is either the torsion energy $E$ or the first eigenvalue of the Dirichlet-Laplacian $\lambda_1$, $N\ge2$ and $\alpha\in(0,N)$.

It is well-known that both the torsion energy and the first eigenvalue of the Dirichlet-Laplacian are minimized, among sets of fixed measure, by the ball. These results, obtained with symmetrization arguments, can be summarized in a scale invariant form as \[
|\Omega|^{-\frac{N}{N+2}}E(\Omega)\ge |B|^\frac{N}{N+2}E(B),\qquad |\Omega|^{\frac{2}{N}}\lambda_1(\Omega)\ge |B|^{\frac{2}{N}}\lambda_1(B),
\]
where $B$ is a generic ball and $|\Om|$ denotes the Lebesgue measure in $\R^N$ of the set $\Om$.
In the literature, they are called \emph{Saint-Venant} and \emph{Faber-Krahn} inequalities, respectively. Both the inequalities are rigid, that is equality holds if and only if  $\Om$ is a ball up to null capacity.
We refer to \cite{hpnew} for a comprehensive background about these problems.

In sharp contrast, the  \emph{Riesz Energy} functional 
\[
V_\alpha(\Omega):=\int_\Omega\int_\Omega\frac{dx\,dy}{|x-y|^{N-\alpha}},
\]
which appears as the second addend in the definition of $\mathcal F$, increases while symmetrizing the set $\Omega$, and it is uniquely (up to a negligible set) maximized by balls \cite[Theorem 3.7]{LL}, leading to a competition while seeking to minimize $\mathcal F$.

\subsection*{Motivation and background}

In recents years the research of quantitative stability of various geometric, functional and spectral inequalities received a great attention, and this gave a strong impulse to the development of the field. In turn this led to a renewed interest in several variational models where a competition between a cohesive term is balanced by a repulsive term. A non-exhaustive list of papers in this field is \cite{aco,bonacini14,f2m3,gnr1,gnr2,km,cp1,cp2,frank15,MurNovRuf,julin14,lu14,FigMag,MurNovRuf}.

Arguably the most famous instance of such variational models is the Gamow \emph{liquid drop} model  introduced in \cite{gamow30} to describe the stability of nuclear matter. Such a model is made up by the sum of a surface perimeter term and a Riesz energy term of a  set $\Omega\subset\R^3$
\[
\mathcal{J}(\Omega):=P(\Omega)+\int_\Omega\int_\Omega\frac{dx\,dy}{|x-y|}.
\]
The usual mathematical questions  about this class of models are:
\begin{enumerate}
\item To investigate existence and non existence of minimizers depending on the values of the $mass$ of competitors, that is, depending on the choice of the volume constraint.
\item To study the regularity of minimizers, if existence holds.
\item To characterize the ball as the unique minimizer as long as the mass is small enough.
\end{enumerate} 
In particular, regarding the liquid drop model, in \cite{cp2} Choksi and Peletier conjectured\footnote{The conjecture was formulated only for $N=3$ but it is commonly extended to any dimension $N\ge2$} that there exists a critical threshold mass $\overline m$ such that minimizers exist only if   $|\Omega|\le\overline m$. Questions (1) and (3) above, as well as such a conjecture, follow the intuitive idea that because of the different scaling of the functionals, if the mass is small  then the perimeter term is dominant, while if the mass is large then the Riesz term dominates, and disconnected configurations are favored. Since the Riesz energy decreases as the connected components of a set are pushed away from each other, this leads to non-existence. In fact, one can show that if the mass is approaching $0$, then the problem reduces to the classical isoperimetric problem. The Choksi-Peletier conjecture,  although being still open in its generality, was partially solved in \cite{km,f2m3, julin14} where the authors  show that there are thresholds $0<m_{small}<m_{big}$ such that the ball is the unique minimizer for $m<m_{small}$ and existence does not occur if $m>m_{big}$. 
The scope of this paper is to begin this kind of analysis when the perimeter is replaced by a spectral functional.
 
\subsection{Main results}
The  main result of the paper is the following. 

\begin{theorem}\label{thm:mainvero}
Let $N\ge 2$, $\alpha\in(1,N)$. There exists $\eps_{\lambda_1}=\eps_{\lambda_1}(N,\alpha)>0$ such that, for all $\eps\leq \eps_{\lambda_1}$, the ball of unitary measure is the unique minimizer for problem 
\begin{equation}\label{eq:minla1}
\min\Big\{\mathcal \lambda_1(\Om)+\eps V_\alpha(\Om) : \Om\subset \R^N,\;|\Om|=1\Big\}.
\end{equation}
\end{theorem}
 
In the case where $\mathcal S=E$ is the torsion energy, we obtain the following weaker result.

\begin{theorem}\label{thm:mainvero2}
Let $N\ge 2$, $\alpha\in(1,N)$. There exists $R_0=R_0(N)$ such that for all $R> R_0$ there exists $\eps_{E}=\eps_{E}(N,\alpha,R)$ such that, for all $\eps\leq \eps_{E}$, the ball is the unique minimizer for problem
\begin{equation}\label{eq:mintor}
\min\Big\{ E(\Om)+\eps V_\alpha(\Om) : \Om\subset \R^N,\;|\Om|=1,\,\Omega\subset B_R\Big\}.
\end{equation}
\end{theorem}
We stress that the value of the geometric constant $R_0$ can be explicitly computed from our proofs.

A  remark concerning the mass constraint is in order.
\begin{remark}\label{rmk:masseps}
A straightforward scaling argument shows  that  there exists a continuous positive function $\varepsilon(m)$ vanishing at the origin and diverging at infinity such that minimizing 
\[
\lambda_1(\Omega)+V_\alpha(\Omega),\qquad|\Omega|=m
\]
is equivalent to minimize the functional 
\[
\lambda_1(\Omega)+\varepsilon(m)V_\alpha(\Omega),\qquad |\Omega|=1,
\]
as for all $t>0$ we have \[
\la_1(t\Om)+V_\alpha(t\Om)=t^{-2}\Big(\la_1(\Om)+t^{N+\alpha+2}V_\alpha(\Om)\Big).
\]
In particular requiring the mass of competitors $m\approx t^N$ to be small is equivalent to require $\eps\approx t^{N+\alpha+2}$ to be small.

Therefore, Theorem~\ref{thm:mainvero} states that for small masses the only minimizer of $\lambda_1+V_\alpha$ is the ball, as long as $\alpha>1$, which is the analogous of the results obtained on the functional $P+V_\alpha$.

For the torsion energy the situation depends on the value of $\alpha$. Indeed for any $t>0$ one has 
\[
E(t\Omega)+V_\alpha(t\Omega)=t^{N+2}\left( E(\Omega)+t^{\alpha-2}V_\alpha(\Omega)\right),
\]
so that small values of $\varepsilon=t^{\alpha-2}$ do correspond to small values of the mass only for $\alpha >2$. 
\end{remark}

The result stated in Theorem \ref{thm:mainvero} is the  spectral analog of the existence results in \cite{km,julin14,f2m3}.
On the other hand, when dealing with the torsion energy, the result needs the additional assumption of equiboundedness of competitors. We believe such an hypothesis to be of technical nature, but its removal  seems  a  challenging task  and we do not solve it in this paper. 
We discuss this issue in the next remark. 
\begin{remark}
The problem of proving the existence of minimizers among generic subsets of $\R^N$ (instead of among equibounded sets) for spectral functionals has been a rather hot topic in the last years.
Regarding  the eigenvalues of the Dirichlet-Laplacian essentially two techniques are available in literature: one developed by Bucur in~\cite{bu}  is based on a concentration-compactness argument mixed together with  regularity results for inward minimizing sets;  the other, proposed by the first author and Pratelli in~\cite{mp}, is based on a De Giorgi type surgery argument.
Seemingly none of these techniques   works while tackling the case of the functional $E+\eps V_\alpha$. 
Even working with  a more direct  surgery-wise technique for the functional $E$ as that used in~\cite[Section~5]{brdeve} seems to fail in our setting. 
Hence we are not able to get rid of the equiboundedness assumption in Theorem~\ref{thm:mainvero2}.
\end{remark}

Restricting the class of Riesz energies to $\alpha\in(1,N)$ seems a deep  problem as well. In fact to show  Theorems \ref{thm:mainvero}, \ref{thm:mainvero2} we need a fine regularity analysis of minimizers (see the discussion below)  where the regularity of the \emph{Riesz potential}
\[
v_\Omega(x):=\int_\Omega \frac{dy}{|x-y|^{N-\alpha}}
\] 
plays a crucial role. If $\alpha\leq 1$, then $v_\Omega$ is at most of class $C^{0,\gamma}$, for some $\gamma\in(0,1]$, which is not enough for our proof to work. 

The third and last result we prove is the following, in which we show that for big values of the mass, minimizers do not exist among sets satisfying uniform density constraints { as long as $\alpha\in(N-1,N)$}.

\begin{definition}
 We say that a set $\Omega$ has the internal $\delta-$ball condition if for any $x\in\partial\Omega$ there exists a ball $B_\delta\subset\Omega$ tangent to $\partial\Omega$ in $x$.
We call $\mathcal U (\delta)$ the class of open sets $\Om\subset \R^N$ that satisfy the internal $\delta$-ball condition.
\end{definition}

\begin{theorem}\label{thm:nonexistballcond}
{Let $\alpha\in  (N-1,N)$. Then there exists $\delta_0\in(0,1)$ such that for any $\delta\in(0,\delta_0)$ there exists
  $\eps_{max}=\eps_{max}(\alpha,N,\delta)$ such that for $\eps\geq \eps_{max}$} both   problems 
\begin{equation}\label{eq:minnonexist}
\inf\left\{ E(\Om)+\eps V_\alpha(\Omega) : \Om\in \mathcal U(\delta),\; |\Om|=1\right\},
\end{equation}
and
\begin{equation}
\inf\left\{ \lambda_1(\Om)+\eps V_\alpha(\Omega) : \Om\in \mathcal U(\delta),\; |\Om|=1\right\},
\end{equation}
do  not admit a minimizer.
\end{theorem}

\subsection{Outline of the proof and structure of the paper}

The proofs of the main results of the paper, Theorem \ref{thm:mainvero} and \ref{thm:mainvero2}, are articulated in two main steps, which we briefly describe here below. First we discuss  the proof of Theorem~\ref{thm:mainvero2}, which covers most of the paper. Then we describe a (completely independent) surgery argument for the functional $\la_1+\eps V_\alpha$. By putting together these two steps, Theorem~\ref{thm:mainvero} follows.

{\bf Strategy of the proof of Theorem~\ref{thm:mainvero2}.}
The proof of Theorem \ref{thm:mainvero2} is quite long and involved and  is inspired by ideas developed in \cite{brdeve,km}.

First of all, we consider a problem without the mass constraint. This step is needed because the techniques from the free boundary regularity that we aim to apply do not work properly in presence of a measure constraint as   perturbations become more difficult to manage.  
Whence  we consider an auxiliary minimization problem of the form
\begin{equation}\label{eq:minGeta}
\min\left\{ \mathcal G_{\eps,\eta}(\Om):=E(\Om)+\eps V_\alpha(\Om)+ f_\eta(|\Om|): \Om\subset B_R\right\},
\end{equation}
where $f_\eta$ is a suitable piecewise linear function which acts as a sort of Lagrange multiplier. This strategy in shape optimization problems was first proposed by Aguilera, Alt and Caffarelli in \cite{agalca}. We point out that without the equiboundedness restriction, at least as long as $\alpha<2$, minimizers of problem \eqref{eq:minGeta} do not exist (see Section \ref{setting}), and the infimum of $\mathcal G_{\eps,\eta}$ diverges toward minus infinity, which somewhat underlines one difficulty while trying to remove the equiboundedness of competitors in Theorem \ref{thm:mainvero2}.
Unfortunately the desired equivalence between~\eqref{eq:minGeta} and~\eqref{eq:mintor} is not straightforward, and we first need to show existence of minimizers of problem \eqref{eq:minGeta}, and some mild regularity (finiteness of the perimeter and density estimates). 
This permits us to show that for suitable values of $\eta$ (again depending on $R$), the minimization of $\mathcal G_{\eps,\eta}$ and the measure constrained minimization of $E+\eps V_\alpha$ are indeed equivalent, for $\eps$ small enough. 

The next key point is therefore to prove a suitable regularity result on the free boundary of an optimal set for~\eqref{eq:minGeta}. 
To get such a regularity we switch to the problem
\[
\min\left\{ \frac12 \int|\nabla u|^2-\int u +\varepsilon V_\alpha(|\{u>0\}|)+f_\eta(|\{u>0\}|) \,:\, u\in H^1_0(B_R)\,\right\},
\]
with the idea of exploiting the regularity theory for $\partial \{u>0\}\cap B_R$, where $u$ is any minimizer of the above problem. Such an analysis is done in the spirit of the seminal work on \emph{free boundary regularity} by Alt and Caffarelli~\cite{alca} . 

The link between this regularity argument and the rigidity of the ball is then 
the \emph{quantitative version} of the Saint-Venant inequality, stating that for any $\Omega\subset\R^N$ there exists a ball $B_r(x)$ of measure $|\Omega|$ such that
\[
|\Omega|^{-\frac{N}{N+2}}E(\Omega)- |B|^{-\frac{N}{N+2}}E(B)\ge \sigma_N \left(\frac{|\Omega\setminus B_r(x)|}{|\Omega|}\right)^2.
\]
This  deep result, recently shown in \cite{brdeve}, together with several of the ideas of its proof, plays a crucial role in our analysis. Indeed by comparing any candidate minimizer with a ball, we show  that for $\varepsilon$ small, minimizers are close in $L^1-$topology to the ball. Whence, exploiting the free boundary regularity analysis, such an $L^1-$ proximity to a ball is improved to a  \emph{nearly spherical} one, stating that the boundary of any minimizer is a small $C^{2,\gamma}-$parametrization on a sphere. At this point,  a perturbative analysis in the class of nearly spherical sets yields to the conclusion (again with the aid of the quantitative Saint-Venant inequality) that the ball is the only minimizer. Beside proving Theorem \ref{thm:mainvero2}, this argument, together with the Kohler-Jobin inequality is enough to get the statement of Theorem \ref{thm:mainvero} among equibounded sets. At this point we only need to show that any minimizing sequence can be chosen to be made up of equibounded sets. This, as mentioned above, is made by means of a surgery-wise argument.

{\bf The surgery argument.}
The strategy we follow is based on that proposed in~\cite{mp} (see also~\cite{buma}) in order to prove existence of minimizers under measure constraint for the $k-$th eigenvalue of the Dirichlet-Laplacian. Nevertheless some differences with respect to \cite{mp} occur.  On the one hand the presence of the repulsive Riesz energy term forces us to work with connected sets. On the other hand we  only deal with the first eigenvalue, thus we do not need to take care of the further difficulty about the orthogonality constraint of the higher eigenfunctions. Furthermore, up to choose $\eps$ small enough, we can deal with sets which are close to the ball in the $L^1-$topology which allows us to simplify the argument.

\medskip
{\bf {Plan of the paper.}}
The paper is organized as follows: in Section \ref{setting} we give the basic definitions and we prove or recall some preliminary results. In Sections \ref{exg}, \ref{MR}, \ref{sect:equivalence} and \ref{HR} we develop the proof of Theorem \ref{thm:mainvero2}, as described above. 
Section~\ref{sect:surgery} is devoted to a surgery argument for the functional $\la_1+\eps V_\alpha$.
 Finally, Sections \ref{maintheorem} and \ref{nonex} contain the proof of Theorems \ref{thm:mainvero} and \ref{thm:nonexistballcond}, respectively.

%\newpage

%\newpage

\section{Setting, notations and some preliminary results}\label{setting}
The ambient space  in this work is $\R^N$, where $N\ge2$ is an integer. With $\Omega$ we denote an open bounded set, unless otherwise stated. We  write $B_r(x)$ to indicate the ball with radius $r$ centered in $x$, and just $B_r$ if the center is $x=0$, while by $B$ we denote just  a generic ball, unless otherwise stated.  
Moreover we set $\omega_N$ the measure of the ball of unit radius in $\R^N$ and the  $N$-dimensional Lebesgue measure of a set $D$ is denoted by $|D|$.  
\subsection{The functionals: definitions and properties}
The problem we deal with is the minimization under volume constraint of the functional
\begin{equation}\label{eq:funct}
\mathcal F_{\alpha,\eps}(\Om):=E(\Om)+\eps  V_\alpha(\Om).
\end{equation}
where $E$ is the torsion energy 
\[
E(\Om):=\min_{u\in H^1_0(\Om)}\frac12\int_\Om|\nabla u|^2-\int_\Om u,
\]
and $ V_\alpha$ is the Riesz potential energy defined for $\alpha\in(0,N)$ as
\[
 V_\alpha(\Om):=\int_\Om\int_\Om \frac{1}{|x-y|^{N-\alpha}}\,dx\,dy.
\]
Some features of these two functionals are in order. First, we remark that  the minimum for the torsion energy functional is attained by a  function $w_\Om$, the torsion function, as long as $\Omega$ has finite measure. The Euler-Lagrange equation of the minimization problem defining $E$ reads as
\[-\Delta w_\Om =1\quad\text{ in $\Om$},\qquad w_\Om\in H^1_0(\Om).
\]
The definition of $E$ together with the equation satisfied by $w_\Om$ leads to the following representation of the torsion energy
\[
E(\Om)=-\frac12\int_\Om w_\Om\leq 0.
\]
By the P\'olya-Sz\"ego inequality (see \cite{PS}) it follows that \[ E(\Om)\ge E(B),
\]
where $B$ is a  ball of $\R^N$ of measure $|\Omega|$, that is: the torsion energy  $E$ is minimized by balls under volume constraint. Moreover the above inequality is \emph{rigid}, in the sense that
  equality holds if and only if $\Om$ is a ball, up to sets of null-capacity. This inequality is addressed as Saint-Venant inequality. The torsion energy $E$ satisfies the following scaling law:
\[
E(t\Om)=t^{N+2}E(\Om),\qquad \text{for all }t>0,
\]
and  it is non-increasing with respect to set inclusion, i.e.\[
\Om_1\subset \Om_2 \qquad \Longrightarrow\qquad E(\Om_1)\geq E(\Om_2),
\]
the inequality being strict as soon as $|\Om_2\setminus \Om_1|>0$.
Therefore we can rewrite the Saint-Venant inequality in the scale invariant form
\begin{equation}\label{eq:saintvenant}
E(\Om)|\Om|^{-\frac{N+2}{N}}\geq E(B)|B|^{-\frac{N+2}{N}}.
\end{equation}

\noindent
About the  Riesz energy functional $V_\alpha$, we note that it  scales as
\[
 V_\alpha(t\Om)=t^{N+\alpha} V_\alpha(\Om).
\]

Moreover, we recall that by Riesz inequality (see \cite[Theorem~3.7]{LL}) $ V_\alpha$ is \emph{maximized} by balls, that is, 
\[
 |\Omega|^{-\frac{N+\alpha}{N}}V_\alpha(\Om)\leq   |B|^{-\frac{N+\alpha}{N}}V_\alpha(B).
\]
Again the inequality is rigid, that is, equality holds if and only if  $\Om$ is a ball up to a negligible set. It is immediate to see that the Riesz potential energy is non-decreasing with respect to set inclusion, that is \[ \Om_1\subset \Om_2,\qquad \text{implies }\qquad  V_\alpha(\Om_1)\leq  V_\alpha(\Om_2),
\]
the inequality being strict if $|\Om_2\setminus \Om_1|>0$.
Alongside the Riesz energy we define the Riesz potential 
\[
v_\Om(x):=\int_\Om\frac{1}{|x-y|^{N-\alpha}}\,dy=\chi_\Om*\frac{1}{|\cdot|^{N-\alpha}}(x),
\]
so that 
\[
 V_{\alpha}(\Om)=\int_\Om v_\Om(x)\,dx.
\]
Notice that $v_\Om(0)$ satisfies
\begin{equation}\label{eq:vomega0}
v_{t\Om}(0)=\int_{t\Om}\frac{1}{|y|^{N-\alpha}}\,dy=t^{\alpha}\int_\Om\frac{1}{|z|^{N-\alpha}}\,dz=t^\alpha v_{\Om}(0).
\end{equation}
The following result, which is a simple refinement of~\cite[Proof of Proposition~2.1]{km}   will be used several times in the paper.
\begin{lemma}\label{le:knupfermuratov}
Let $\alpha\in (0,N)$, $\Om,F\subset \R^N$ be two measurable sets, with finite {measure}, such that $\Om\Delta F\subset B_R(0)$, for some $R>0$. Then it holds 
\begin{equation}\label{eq:knmuI}
 V_\alpha(F)- V_\alpha(\Om)\leq C_0|\Om\Delta F|\,\Big[|\Om|^{\frac{\alpha}{N}}+|F|^{\frac{\alpha}{N}}\Big],
\end{equation} 
for some constant $C_0=C_0(N,\alpha)>1$.
\end{lemma}
\begin{proof}
%Let us set, for any couple of measurable sets $A$ and $G$
%\[
%V_\alpha(A,G):=\int_A\int_G\frac{1}{|x-y|^{N-\alpha}}\,dx\,dy=\int_Av_G(x)\,dx\geq 0.
%\]

{
First  we compute
\begin{equation}\label{eq:Iint}
\begin{split}
V_\alpha(F)- V_\alpha(\Om) &=\int_{\R^N}\int_{\R^N}\frac{\chi_F(x)(\chi_F(y)-\chi_\Omega(y))}{|x-y|^{N-\alpha}}\,dxdy+\int_{\R^N}\int_{\R^N}\frac{\chi_\Omega(y)(\chi_F(x)-\chi_\Omega(x))}{|x-y|^{N-\alpha}}\,dxdy\\
&=\int_{F\setminus \Omega}(v_F(x)+v_\Omega(x))\,dx-\int_{\Omega\setminus F}(v_F(x)+v_\Omega(x))\,dx\leq\int_{\Om\Delta F}(v_F+v_\Om).
\end{split}
\end{equation}
}
We can now observe that, as a consequence of Riesz inequality (see~\cite[Lemma~2.3]{fuprariesz}) and the rescaling of $v_\Om(0)$, see~\eqref{eq:vomega0}, we get
\[
\begin{split}
&{\int_{\Omega\Delta F}v_\Omega(x)\,dx=}\int_{\Om\Delta F}\int_\Om\frac{1}{|x-y|^{N-\alpha}}\,dy\,dx\\
&\leq \int_{\Om\Delta F}\int_{{\widetilde B(x)}}\frac{1}{|x-y|^{N-\alpha}}\,dy\,dx \\
&= |\Om\Delta F|\int_{{\widetilde B}}\frac{1}{|z|^{N-\alpha}}\,dz=|\Om\Delta F|\,|\Om|^{\frac{\alpha}{N}}\int_B\frac{1}{|z|^{N-\alpha}}\,dz,
\end{split}
\]
{where $\widetilde B(x)$ and $\widetilde B$ are balls of measure $|\widetilde B|=|\widetilde B(x)|=|\Omega|$ centered at $x$ and at the origin respectively, while $B$ is the ball of measure one centered at the origin}.
The same computation holds also for $\int_{\Om\Delta F}v_F\,dx$.
In conclusion we have \[
 V_\alpha(F)- V_\alpha(\Om)\leq \int_{\Om\Delta F}(v_F+v_\Om)\leq C_0|\Om\Delta F|\,\Big[|\Om|^{\frac{\alpha}{N}}+|F|^{\frac{\alpha}{N}}\Big],
\]
where $C_0(N,\alpha):=\int_B\frac{1}{|z|^{N-\alpha}}\,dz<+\infty$ as $\alpha>0$.
\end{proof}
We conclude this subsection recalling one of the main tool we  exploit to solve problem \eqref{eq:main}:  the sharp quantitative version of the Saint-Venant inequality, which was first proved as a intermediate result in~\cite[Proof of the Main Theorem]{brdeve}.%
\begin{theorem}\label{thm:quantitative}
There exists a constant $\sigma=\sigma(N)$, such that, for all open sets { with finite measure} $\Om\subset \R^N$, we have 
\begin{equation}\label{eq:quantsv}
E(\Om)|\Om|^{-1-\frac2N}-E(B)|B|^{-1-\frac2N}\geq \sigma \mathcal A(\Om)^2,
%\frac{E(\Om)-E(B_m)}{E(\Omega)}\geq \sigma \mathcal A(\Om)^2,
\end{equation}
for any ball $B$, 
where 
\begin{equation}\label{eq:fraenkel}
\mathcal A(\Omega):=\inf \left\{ \frac{|\Omega\Delta B(x)|}{|\Omega|}\,:\, x\in\R^N, \text{$B(x)$ is a ball of measure $|\Omega|$}\right\},
\end{equation}
is the Fraenkel asymmetry.
\end{theorem}

The last functional involved in our work is the first eigenvalue of the Dirichlet-Laplacian acting on an open and bounded set $\Om\subset \R^N$. We recall its variational definition given as the minimum of the so-called Rayleigh quotient:\[
\la_1(\Om):=\min_{\varphi\in H^1_0(\Om)}\frac{\int_{\Om}|\nabla \varphi|^2}{\int_{\Om}\varphi^2},
\]
we call $u\in H^1_0(\Om)$ the function attaining the minimum, which is the eigenfunction corresponding to $\la_1(\Om)$ and that solves the PDE\[
\begin{cases}
-\Delta u=\la_1(\Om)u,\qquad &\text{in }\Om,\\
u\in H^1_0(\Om).&
\end{cases}
\]
The monotonicity and scaling properties of the eigenvalue follow immediately from its definition:\[
\begin{split}
\la_1(t\Om)=t^{-2}\la_1(\Om),\qquad &\text{for all }t>0,\\
\Om_1\subset \Om_2\qquad \Longrightarrow&\qquad \la_1(\Om_1)\geq \la_1(\Om_2).
\end{split}
\]

We finally recall the sharp quantitative Faber-Krahn inequality for the first eigenvalue of the Dirichlet-Laplacian, that was first proved in~\cite[Main Theorem]{brdeve}.
\begin{theorem}\label{thm:quantitativefk}
There exists a positive constant $\widehat \sigma=\widehat \sigma(N)$ such that for all open set $\Om\subset\R^N$ with finite measure we have \[
|\Om|^{2/N}\la_1(\Om)-|B|^{2/N}\la_1(B)\geq \widehat \sigma \mathcal A(\Om)^2,
\]
where $B$ is a generic ball and $\mathcal A$ the Fraenkel asymmetry.
\end{theorem}

\subsection{Quasi-open sets and the minimization problem}
Let us recall the notion of capacity and of quasi-open set.
\begin{definition}\label{def:cap}
For every subset $D$ of $\R^N$, the capacity of $D$ in 
$\R^N$ is defined as 
\begin{multline*}
\cp(D):=\inf\,\biggl\{
\int \left(|\nabla u|^2+|u|^2\right)\,dx:\,\, 
u\in H^1(\R^N)\,, 
\biggr. \\ \biggl.
\text{$0\leq u\leq 1$ $\mathcal L^N$-a.e. on $\R^N$,
$u = 1$ $\mathcal L^N$-a.e. on an open set containing $D$}\biggr\}\,.
\end{multline*}

We say that a property $\mathcal P(x)$ 
holds \emph{$\cp$-quasi-everywhere in $D$}, if it holds for 
all $x\in D$ except at most a set of zero capacity, and in this case we    write q.e. in $D$.
A subset $A$ of $\R^N$ is said to be \emph{quasi-open} if  
for every $\eps>0$  there exists an open subset $\omega_\eps$ 
of $\R^N$ such that $\cp(\omega_\eps) <\eps$ and 
$A\cup \omega_\eps$ is open. 
\end{definition}
The notion of capacity is strictly related to spectral functionals such as the torsion energy and the first eigenvalue of the Dirichlet-Laplacian.
In particular, one can not consider to be equivalent, a priori, two open (or quasi-open) sets which differ for a generic negligible set. Indeed for any open set $\Omega$ it is possible to construct a sequence of subsets $\Om_n\subset\Omega$ of measure $|\Omega_n|=|\Om|$ with $E(\Om_n)<1/n$ for all $n\in \N$.
For example take $\Om=(0,1)^N$ and let $\{r_i\}_{i\in \N}$ be an enumeration of the rationals in $(0,1)$. Then, as $\cp((0,1)^{N-1})>0$, it is possible to find $k_n$ so that \[
\Om_n=\Om\setminus \left\{(0,1)^{N-1}\times \bigcup_{i=1}^{k_n}r_i\right\},\qquad \text{with }E(\Om_n)\leq \frac1n,
\]
and $|\Om_n|=|\Om|$.

\begin{definition}
A function $u: \R^N\rightarrow \overline \R$ is said to be 
\emph{quasi-continuous} if for every $\eps>0$ 
there exists an open subset $\omega_\eps$ of $\R^N$ with 
$\cp(\omega_\eps)<\eps$ such that 
$u\bigl|_{\R^N\setminus \omega_\eps}$ is continuous.
\end{definition}
For every $u\in H^1(\R^N)$, there exists a 
Borel and quasi-continuous representative 
$\tilde{u}:\R^N\rightarrow\R$ of~$u$ and, if $\tilde{u}$ and 
$\hat{u}$ are two quasi-continuous representatives of 
the same function $u$, then we have $\tilde{u}=\hat{u}$ q.e. in~$\R^N$.
From now on for every $u\in H^1(\R^N)$,
we   consider only its quasi-continuous 
representative. In this setting, we are able to provide a more general definition of the space $H^1_0(\Om)$, which coincide with the usual one as soon as $\Om$ is open, but that is suitable also for measurable sets (and quasi-open sets in particular).
\begin{definition}\label{H10}
If $A$ is a quasi-open subset of $\R^N$, we set
\[
H^1_0(A):=\left\{u\in H^1(\R^N) :\,\,
\text{$u=0$ q.e. in $\R^N\setminus A$}\right\}\,.
\]
\end{definition}

%\\
%QUASI APERTI
%\\

It is nowadays standard to perform the minimization of functionals such as $\mathcal F_{\alpha,\eps}$ in the class of quasi-open sets. 
As it can be noted, in the definition of the torsion energy and of the first eigenvalue of the Dirichlet-Laplacian, only the space $H^1_0(\Om)$ was really needed and therefore, once we have a definition which is suitable for quasi-open sets, we can work with them with no additional worries. On the other hand, the Riesz energy is well defined even for measurable sets, therefore there are no problems on its side. 

As it is common in the Calculus of Variations, after finding a minimizer in the larger class of quasi-open sets, we will try later to restore the regularity of minimizers (and in particular, show that they are open). 

We are now in position to properly define the problem we deal with in a large part of this paper. Let $R>\left(\frac{1}{\omega_N}\right)^{1/N}$, so that a ball of radius $R$ has measure greater than $1$. Then we consider the problem
\begin{equation}\label{eq:main}
\min\left\{\mathcal F_{\alpha,\eps}(A) : A\subset \R^N,\; \text{quasi-open},\; |A|=1,\, A\subset B_R\right\}.
\end{equation} 
From now on, we tacitly deal with quasi-open sets, unless otherwise stated.

\subsection{Some notions of geometric measure theory}
We give here some measure theoretic notions which will be used throughout the paper. Comprehensive references for this section are \cite{afp,maggi}. The \emph{measure theoretic} perimeter (or De Giorgi perimeter) of a measurable set $E$ is the quantity
\[
P(E)=\sup\left\{\int_E\rm{div}\zeta\, :\,  \zeta\in C^1_c(\R^N,\R^N),\,\|\zeta\|_{C^0}\le1 \right\}.
\]
We say that $E$ is a \emph{set of finite perimeter} or \emph{Caccioppoli set} if $P(E)<+\infty$, that is if $\chi_E$ is a function of bounded variation \cite{afp}, and with $\nabla \chi_E$ we indicate the distributional derivative of $\chi_E$. Notice that if $E$ is Lipschitz regular, by divergence theorem,
\[
P(E)=\mathcal H^{N-1}(\partial E),
\]
where $\mathcal H^{k}$ stands for the $k-$dimensional Hausdorff measure, $k\in[0,N]$.

For any Lebesgue  measurable set $E$ and $t\in[0,1]$ we define the quantities
\[
E^{t}=\left\{ x\in\R^N\,:\, \limsup_{r\to0} \frac{|E\cap B_r(x)|}{|B_r(x)|}=t\,\right\},
\]
and the \emph{essential boundary of $E$} as 
\[
\partial^M E:=\R^N\setminus(E^{0}\cup E^1).
\]
Beside the essential boundary we call \emph{reduced boundary} the set
\[
\partial^*E:=\left\{  x\in\R^N\,:\, \nu_E(x):=\lim_{r\to0}\frac{\int_{B_r(x)}\nabla \chi_E}{\int_{B_r(x)}|\nabla \chi_E|}\,\text{ exists and is a unit vector }  \right\}.
\]
The quantity $\nu_E(x)$ in the definition of $\partial^*E$ is the measure theoretic normal of $\partial E$ at the point $x$, whenever it is well defined.
By results of Federer and De Giorgi \cite{maggi} for sets of finite perimeter it holds
\[
P(E)=\mathcal H^{N-1}(\partial^*E)=\mathcal H^{N-1}(\partial^ME).
\]
In particular for a set of finite perimeter we have  $\partial^*E\subset E^{1/2}\subset \partial^ME$ and  $\mathcal H^{N-1}(\partial^ME\setminus\partial^*E)=0$. Eventually, for any $x\in\partial^*E$ the blow up of the boundary of $E$ converges  in $L^1$ to an hyperplane orthogonal to $\nu_E(x)$, that is
\[
\frac{E-x}{r}\to\left\{y\in\R^N\,:\, y\cdot  \nu_E(x)\ge0\right\}
\] 
as $r\to0$.

%\newpage
\section{An existence result for an auxiliary problem}\label{exg}

Let $\eta\in(0,1)$ and consider the function \[
f_\eta\colon \R^+\rightarrow \R,\qquad f_\eta(s)=
\begin{cases}
\eta(s-1),\qquad \text{if }s\leq 1,\\
\frac{1}{\eta}(s-1),\qquad \text{if }s\geq 1.
\end{cases}
\]
It is easy to check that, for all $0\leq s_2\leq s_1,$ we have that
\begin{equation}\label{eq:propfeta}
\eta(s_1-s_2)\leq f_\eta(s_1)-f_\eta(s_2)\leq \frac{1}{\eta}(s_1-s_2).
\end{equation}
We introduce then the  functional
 \[
\mathcal G_{\eps,\eta}(\Om):=\mathcal F_{\alpha,\eps}(\Om)+f_\eta(|\Om|),
\]
and, for $R>\omega_N^{-1/N}$, the  minimization problem
\begin{equation}\label{eq:minGeta1}
\min\left\{\mathcal G_{\eps,\eta}(\Om) : \Om\subset B_R,\text{$\Om$ quasi-open}\right\}.
\end{equation}
In Section \ref{sect:equivalence} we will show that such minimization problem and the  minimization problem \eqref{eq:main} are equivalent. To do that we first have to prove   existence and some mild regularity of minimizers of $\G$. We begin by showing a lower bound for $\mathcal G_{\eps,\eta}$ on equibounded sets.
\begin{lemma}\label{le:boundbelowR}
Let $\alpha\in (0,N)$, $R>\omega_N^{-1/N}$ and $\eta,\eps\in (0,1)$.
Then, for all quasi-open $\Om\subset B_R$, we have \[
\mathcal G_{\eps,\eta}(\Om)\geq \omega_N^{\frac{N+2}{N}} E(B)R^{N+2}-\eta\geq \omega_N^{\frac{N+2}{N}} E(B)R^{N+2}-1,
\]
where $B$ is any ball of measure $1$.
\end{lemma}

\begin{proof}
Since $\Omega\subset B_R$, by the monotonicity of $E$, its scaling properties and the positivity of $ V_\alpha$ we get 
\[
E(\Om)+\eps V_\alpha(\Om)\ge \omega_N^{\frac{N+2}{N}} E(B)R^{N+2}.
\]
On the other hand, if $|\Om|\geq 1$ then
 \[
f_\eta(|\Om|)\geq 0,
\]
while if $|\Om|<1$ then \[
f_\eta(|\Om|)=\eta(|\Om|-1)\geq -\eta,
\]
and the conclusion easily follows.
\end{proof}
%
%\subsection{An existence result for the unconstrained functional}

The following existence result for the unconstrained functional $\mathcal G_{\eps,\eta}$ is an adaptation of~\cite[Lemma~4.6]{brdeve}, which is in turn inspired by \cite[Theorem~2.2 and Lemma~2.3]{bu}. 

\begin{lemma}\label{le:existminG}
Let $\alpha\in(0,N)$, $\eta\in(0,1)$,  $\eps\in(0,1)$  and let  $R>\omega_N^{-1/N}$.
There exists a minimizer in the class of quasi-open sets for problem~\eqref{eq:minGeta1}.
Moreover all minimizers have perimeter uniformly bounded  by a constant depending  on  $N,R,\eta$.
\end{lemma}
\begin{proof}
Let $(\Om_n)\subset B_R$ be a minimizing sequence, with \[
\mathcal G_{\eps,\eta}(\Om_n)\leq \inf\left\{\mathcal G_{\eps,\eta}(\Om) : \Om\subset B_R,\; \text{quasi-open}\right\}+\frac{1}{n}.
\] 
Let $u_n$ be the torsion function of $\Om_n$, so that $\Om_n=\{u_n>0\}$ and
let $t_n=1/\sqrt{n}$. We define \[
\widetilde \Om_n:=\{u_n>t_n\}.
\]  
We have \[
\mathcal G_{\eps,\eta}(\Om_n)\leq \mathcal G_{\eps,\eta}(\widetilde \Om_n)+\frac{1}{n},
\]
which, since the torsion function of $\widetilde \Om_n$ is precisely $(u_n-t_n)_+$, reads as
 \[
\begin{split}
&\frac12\int_{\{u_n>0\}}|\nabla u_n|^2-\int_{\{u_n>0\}}u_n+\eps V_\alpha(\Om_n)+f_\eta(|\{u_n>0\}|)\\
&\leq \frac12\int_{\{u_n>t_n\}}|\nabla u_n|^2-\int_{\{u_n>t_n\}}(u_n-t_n)_++\eps V_\alpha(\widetilde \Om_n)+f_\eta(|\{u_n>t_n\}|)+\frac{1}{n}.
\end{split}
\]
Noting that 
\begin{equation}\label{eq:star}
\int_{\{u_n>0\}}u_n-\int_{\{u_n>t_n\}}(u_n-t_n)\leq t_n|\{u_n>0\}|,
\end{equation}
recalling the property~\eqref{eq:propfeta} of $f_\eta$ and the monotonicity of $ V_\alpha$,  the above inequality yields 
\begin{equation}\label{eq:e1}
\frac12\int_{\{0<u_n<t_n\}}|\nabla u_n|^2+\frac{\eta}{2}|\{0<u_n<t_n\}|\leq t_n|\{u_n>0\}|+\frac1n\leq t_n|B_R|+\frac1n.
\end{equation}
On the other hand, since $\eta<1$, using coarea formula, the arithmetic geometric mean inequality and~\eqref{eq:e1}, we obtain\[
\begin{split}
&\eta\int_{0}^{t_n}P(\{u_n>s\})\,ds=\eta\int_{\{0<u_n<t_n\}}|\nabla u_n|\,dx\\
&\leq \frac{\eta}{2}\int_{\{0<u_n<t_n\}}|\nabla u_n|^2\,dx+\frac{\eta}{2}|\{0<u_n<t_n\}|\leq t_n|B_R|+\frac{1}{n}.
\end{split}
\]
Thanks to the choice of $t_n=1/\sqrt{n}$, we can find a level $0<s_n<1/\sqrt{n}$ such that the sets $W_n:=\{u_n>s_n\}$ satisfy\[
P(W_n)\leq \frac{2\eta}{\eta t_n}\int_0^{t_n}P(\{u_n>s\})\,ds\leq \frac{2|B_R|}{\eta}+\frac{2}{\eta t_n n}\leq C(N,R,\eta)+\frac{2}{\eta\sqrt{n}}.
\]
It is easy to check that $(W_n)$ is still a  minimizing sequence for problem~\eqref{eq:minGeta1}:
\begin{equation}\label{eq:e2}
\begin{split}
&\mathcal G_{\eps,\eta}(W_n)\\
&=\frac12\int_{\{u_n>s_n\}}|\nabla u_n|^2-\int_{\{u_n>s_n\}}(u_n-s_n)+\eps V_\alpha(\{u_n>s_n\})+f_\eta(|\{u_n>s_n\}|)\\
&\leq \mathcal G_{\eps,\eta}(\Om_n)+s_n|\{u_n>0\}|+f_\eta(|\{u_n>s_n\}|)-f_\eta(|\{u_n>0\}|)\\
&\leq \mathcal G_{\eps,\eta}(\Om_n)+\frac{|B_R|}{\sqrt{n}}-\eta|\{0<u_n<s_n\}|\leq \mathcal G_{\eps,\eta}(\Om_n)+\frac{|B_R|}{\sqrt{n}},
\end{split}
\end{equation}
where we have also used the monotonicity of $ V_\alpha$ and property~\eqref{eq:star} with $s_n$ in place of $t_n$. 
Moreover, since the sets of the sequence $(W_n)_{n\in\mathbb N}$ have equibounded perimeter, there exists a Borel set $W_\infty$ such that (up to pass to subsequences) \[
W_n\rightarrow W_\infty,\text{ in }L^1,\qquad P(W_\infty)\leq C(N,R,\eta).
\]
On the other hand, the torsion function of $W_n$, that is $w_n=(u_n-s_n)_+$, is equibounded in $H^1(B_R)$. In fact, by Lemma~\ref{le:boundbelowR}, $\mathcal G_{\eps,\eta}$ is (uniformly) bounded from below and so\[
C(N,R)\leq \mathcal G_{\eps,\eta}(W_n)=-\frac12\int_{W_n}|\nabla w_n|^2+\eps V_\alpha(W_n)+f_\eta(|W_n|),
\]
which implies, \[
\frac12\int_{W_n}|\nabla w_n|^2\leq -C(N,R)+\eps V_\alpha(B_R)+\frac{1}{\eta}|B_R|.
\]
Hence, up to subsequences, there is $w\in H^1_0(B_R)$ such that \[
w_n\rightarrow w,\qquad \text{strongly in }L^2(B_R)\text{ and weakly in }H^1_0(B_R).
\]
We set $W:=\{w>0\}$, and recall that we are identifying $w$ with its quasi-continuous representative. Thus \[
\chi_W(x)\leq \liminf_{n\rightarrow \infty}\chi_{W_n}(x)=\chi_{W_\infty}(x),\qquad \text{ for a.e. }x\in B_R,
\]
hence $|W\setminus W_\infty|=0$, that is $W\subset W_\infty$ up to a negligible set.
We now observe that $ V_\alpha$ and $f_\eta$ are continuous with respect to the $L^1$ convergence of sets, while the first integral in the torsion energy is lower semicontinuous with respect to the weak $H^1$ and the second one with respect to the  strong $L^1$ convergence. We can therefore pass to the limit in~\eqref{eq:e2} and obtain \[
\begin{split}
&E(W)+\eps  V_\alpha(W_\infty)+f_\eta(|W_\infty|)\leq \frac12\int_{B_R}|\nabla w|^2-\int_{B_R}w+\eps V_\alpha(W_\infty)+f_\eta(|W_\infty|)\\
&\leq \liminf_{n}\mathcal G_{\eps,\eta}(W_n)=\inf_{\Om\subset B_R}\mathcal G_{\eps,\eta}(\Om)\leq E(W)+\eps V_\alpha(W)+f_\eta(|W|).
\end{split}
\] 
On the other hand, using again the monotonicity of $ V_\alpha$, we have \[
\eta|W_\infty\setminus W|=\eta(|W_\infty|-|W|)\leq f_\eta(|W_\infty|)-f_\eta(|W|)\leq \eps( V_\alpha(W)- V_\alpha(W_\infty))\leq 0,
\]
thus $|W_\infty\setminus W|=0$, which entails $W=W_\infty$ a.e. and this is the desired minimizer for problem~\eqref{eq:minGeta1}.
\end{proof}

\noindent We conclude this section with a  result concerning a property of the minimizers of $\G$ which will be useful later.
\begin{lemma}\label{le:infneg}
Let $R>\omega_N^{-1/N}$, $\alpha\in(0,N)$ and $B$ a ball of measure $1$.
There exist  a constants  $\eps_0=\eps_0(N,\alpha)>0$ and $\eta_0=\eta_0(N,\alpha)>$  such that, if $\eta\leq \eta_0$ and $\eps\leq \eps_0$, then for any minimizer $\widehat \Om$   of problem~\eqref{eq:minGeta1} we have \[
E(\widehat \Om)\leq \frac{E(B)}{4}<0.
\] 
\end{lemma}
\begin{proof}
The existence of an optimal set $\widehat \Om$ follows from Lemma~\ref{le:existminG}.
If $|\widehat \Om|\geq 1$, then we have, calling $B$ a ball of unit measure, \[
E(\widehat \Om)\leq E(\widehat \Om)+\eps  V_\alpha(\widehat \Om)+\frac{1}{\eta}(|\widehat \Om|-1)\leq E(B)+\eps  V_\alpha(B)\leq \frac{E(B)}{4}<0,
\]
by minimality of $\widehat \Om$ and as soon as we take {$\eps\leq \eps_0:=\frac{-E(B)}{4V_\alpha(B)}$. }

On the other hand, if $|\widehat \Om|<1$, using again the optimality of $\widehat \Om$ we have \[
E(\widehat \Om)+\eta(|\widehat \Om|-1)\leq E(\widehat \Om)+\eps V_\alpha(\widehat \Om)+\eta(|\widehat \Om|-1)\leq E(B)+\eps  V_\alpha(B), 
\]
that is,
\[
E(\widehat \Om)\leq E(B)+\eps  V_\alpha(B)+\eta\leq \frac{E(B)}{4}<0,
\]
as soon as $\eps\leq \eps_0$ and $\eta\leq \eta_0=\frac{-E(B)}{4}$.
\end{proof}
%

%\newpage
\section{First regularity properties of minimizers of the unconstrained problem}\label{MR}
In this Section we essentially follow the approach of~\cite[Section~4]{brdeve}, which is in turn based on the seminal paper by Alt and Caffarelli~\cite{alca}, to prove density estimates, and Lipschitz regularity of the torsion function of minimizers for  Problem \eqref{eq:minGeta1}.

The keystone idea is that we can pass from a functional defined on the class of quasi-open sets, to another  defined on functions. In fact, for any $\Om\subset B_R$ quasi-open, calling $u$ its torsion function, we have that \[
\mathcal G_{\eps,\eta}(\Om)=\mathcal G_{\eps,\eta}(\{u>0\}).
\]
Moreover, if $\Om_{{\eps,\eta}}$ is optimal for problem~\eqref{eq:minGeta1}, using the definition and minimality properties of its torsion function $u_{{\eps,\eta}}$, we have that, for all $v\in H^1_0(B_R)$, 
\begin{equation}\label{eq:minpropu}
\begin{split}
&\frac{1}{2}\int |\nabla u_{{\eps,\eta}}|^2-\int u_{{\eps,\eta}}+\eps V_\alpha(\{u_{{\eps,\eta}}>0\})+f_\eta(|\{u_{{\eps,\eta}}>0\}|)\\
&\leq \frac{1}{2}\int |\nabla v|^2-\int  v+\eps V_\alpha(\{v>0\})+f_\eta(|\{v>0\}|).
\end{split}
\end{equation}
\begin{remark}
In this section, we stress that instead of working on optimal sets for problem~\eqref{eq:minGeta1}, we focus on functions optimal for problem~\eqref{eq:minpropu}. Clearly if $u$ is optimal for problem~\eqref{eq:minpropu}, then it must be the torsion function of $\{u>0\}$, therefore the two formulations are equivalent.
\end{remark}
By Lemma~\ref{le:knupfermuratov} we get that  $u_{{\eps,\eta}}$ behaves like a \emph{quasi-minimizer}\footnote{This terminology is borrowed by the theory of quasi-minimizers for the perimeter, see \cite[Chapter~3]{maggi}.} of a free boundary-type problem, that is
\begin{equation}\label{eq:uquasiminimizer}
\begin{split}
&\frac{1}{2}\int |\nabla u_{{\eps,\eta}}|^2-\int u_{{\eps,\eta}}+f_\eta(|\{u_{{\eps,\eta}}>0\}|)\\
&\leq \frac{1}{2}\int |\nabla v|^2-\int  v+f_\eta(|\{v>0\}|)\\
&+C\eps|\{u_{{\eps,\eta}}>0\}\Delta\{v>0\}|\;\Big[|\{u_{\eps,\eta}>0\}|^{\frac{\alpha}{N}}+|\{v>0\}|^{\frac{\alpha}{N}}\Big],
\end{split}
\end{equation}
for all $v\in H^1_0(B_R)$ and with a constant $C$ depending only on $N,\alpha$.
Since $v,u_{\eps,\eta}\in H^1_0(B_R)$, from~\eqref{eq:uquasiminimizer} ensues \[
\begin{split}
&\frac{1}{2}\int |\nabla u_{{\eps,\eta}}|^2-\int u_{{\eps,\eta}}+f_\eta(|\{u_{{\eps,\eta}}>0\}|)\\
&\leq \frac{1}{2}\int |\nabla v|^2-\int  v+f_\eta(|\{v>0\}|)+2C|B_R|^{\frac{\alpha}{N}}\eps|\{u_{{\eps,\eta}}>0\}\Delta\{v>0\}|.
\end{split}
\]

This quasi-minimality property does not provide any new information by itself and we need to take advantage of the (smallness of the) parameter $\eps$, since the volume term is not in general of lower order. 
We also observe that if $v\in H^1_0(B_R)$ is such that 
\[
\{v>0\}\subset \{u_{{\eps,\eta}}>0\},
\]
then inequality \eqref{eq:minpropu} together with  the monotonicity of $ V_\alpha$ entails that
\begin{equation}\label{eq:uquasiminimizerbis}
\begin{split}
&\frac{1}{2}\int |\nabla u_{{\eps,\eta}}|^2-\int u_{{\eps,\eta}}+f_\eta(|\{u_{{\eps,\eta}}>0\}|)\\
&\leq \frac{1}{2}\int |\nabla v|^2-\int  v+f_\eta(|\{v>0\}|),
\end{split}
\end{equation}
and we stress the fact that the parameter $\alpha$ does not appear in this formulation.
Therefore, it should not be surprising that in the next Lemma~\ref{le:lemma4.9} the constants ( as for example $K_0,\rho_0$) do not depend on $\alpha$.

We continue our analysis of the regularity of minimizers with the following non-degeneracy lemma. Its proof, which we provide for the sake of completeness, is basically a rewriting of~\cite[Lemma~4.9]{brdeve},  in turn inspired by~\cite[Lemma~3.4]{alca}.
\begin{lemma}\label{le:lemma4.9}
Let $\alpha\in(0,N)$, $R>0$, $\eta\in (0,1)$, $\eps\in(0,1)$ and $\Om$ be an optimal set for the problem
\begin{equation}\label{eq:geps}
\min{\Big\{\mathcal G_{\eps,\eta}(A) : A\subset B_R,\text{ quasi-open}\Big\}},
\end{equation}
we call $u\in H^1_0(\Om)$ its torsion function.
For every $\kappa\in (0,1)$, there are positive constants $K_0,\rho_0$ depending only on $\kappa, \eta, N$ such that the following assertion holds: if $\rho\leq \rho_0$ and $x_0\in B_R$, then 
\begin{equation}\label{eq:nondegmean}
\mean{\partial B_\rho(x_0)\cap B_R}{u\,d\mathcal H^{N-1}}\leq K_0 \rho\,\,\,\,
\Longrightarrow \,\,\,\,u\equiv0 \text{ in }\,\,B_{\kappa\rho}(x_0)\cap B_R.
\end{equation}
\end{lemma}
\begin{proof}
Without loss of generality, we fix $x_0=0$. We also extend $u$ to zero outside $B_R$, so that it satisfies $-\Delta u\leq 1$ in $\R^N$ in weak sense. Then the function \[
x\mapsto u(x)+\frac{|x|^2-\rho^2}{2N}
\]
is subharmonic in $B_\rho$.
Thus, for every $\kappa\in(0,1)$, there exists $c=c(\kappa,N)$ such that 
\begin{equation}\label{eq:4.20}
\delta_\rho:=\sup_{B_{\sqrt{\kappa}\rho}}u\leq c\left(\mean{\partial B_\rho\cap B_R}{u\,d\mathcal H^{N-1}}+\rho^2\right)\leq c(K_0\rho+\rho^2).
\end{equation}
Let now $w$ be the solution of \begin{equation}\label{www}
\begin{cases}
-\Delta w=1,\qquad &\text{in }B_{\sqrt{\kappa}\rho}\setminus B_{\kappa \rho},\\
w=\delta_\rho,\qquad &\text{on }\partial B_{\sqrt{\kappa}\rho},\\
w=0,\qquad &\text{on }B_{\kappa\rho}.
\end{cases}
\end{equation}
By definition, $w\geq u$ on $\partial B_{\sqrt{\kappa}\rho}$, therefore the function \[
v=
\begin{cases}
u,\qquad &\text{in }\R^N\setminus B_{\sqrt{\kappa}\rho},\\
\min\{u,w\},\qquad &\text{in }B_{\sqrt{\kappa}\rho},
\end{cases}
\]
satisfies \[
\{v>0\}\subset \{u>0\},\qquad \{v>0\}\setminus B_{\sqrt{\kappa}\rho}=\{u>0\}\setminus B_{\sqrt{\kappa}\rho}.
\]
Since $v\in H^1_0(B_R)$  inequality~\eqref{eq:uquasiminimizerbis} gives\[
\begin{split}
&\frac{1}{2}\int_{B_{\sqrt{\kappa}\rho}} |\nabla u|^2-\int_{B_{\sqrt{\kappa}\rho}} u+f_\eta(|\{u>0\}|)\\
&\leq \frac{1}{2}\int_{B_{\sqrt{\kappa}\rho}} |\nabla v|^2-\int_{B_{\sqrt{\kappa}\rho}}  v+f_\eta(|\{v>0\}|).
\end{split}
\]
We note that $v=0$ in $B_{\kappa\rho}$, therefore, using also~\eqref{eq:propfeta}, \[
\begin{split}
\frac{\eta}{2}|\{u>0\}\cap B_{\kappa\rho}|&\leq \eta|(\{u>0\}\setminus \{v>0\})\cap B_{\sqrt{\kappa}\rho}|\\
&\leq f_\eta(|\{u>0\}|)-f_\eta(|\{v>0\}|).
\end{split}
\]
Thanks to the two inequalities above and the definition of $v$, we can infer 
\begin{equation}\label{eq:4.21}
\begin{split}
&\frac12\int_{B_{\kappa\rho}}|\nabla u|^2-\int_{B_{\kappa\rho}}u+\frac{\eta}{2}|\{u>0\}\cap B_{\kappa\rho}|\\
&\leq \frac12\int_{B_{\kappa\rho}}|\nabla u|^2-\int_{B_{\kappa\rho}}u+f_\eta(|\{u>0\}|)-f_\eta(|\{v>0\}|)\\
&\leq \frac12\int_{B_{\sqrt{\kappa}\rho}\setminus B_{\kappa\rho}}(|\nabla v|^2-|\nabla u|^2)-\int_{B_{\sqrt{\kappa}\rho}\setminus B_{\kappa\rho}}(v-u)\\
&\leq \int_{(B_{\sqrt{\kappa}\rho}\setminus B_{\kappa\rho})\cap\{u>w\}}(|\nabla w|^2-\nabla u\cdot\nabla w)-\int_{(B_{\sqrt{\kappa}\rho}\setminus B_{\kappa\rho})\cap \{u>w\}}(w-u).
\end{split}
\end{equation} 
On the other hand testing \eqref{www} with $(u-w)_+$ and integrating over ${B_{\sqrt{\kappa}\rho}\setminus B_{\kappa\rho}}$, we obtain 
\begin{equation}\label{eq:4.22}
\int_{(B_{\sqrt{\kappa}\rho}\setminus B_{\kappa\rho})\cap \{u>w\}}(|\nabla w|^2-\nabla u\cdot\nabla w)-\int_{(B_{\sqrt{\kappa}\rho}\setminus B_{\kappa\rho})\cap \{u>w\}}(w-u)=\int_{\partial B_{\kappa\rho}}\frac{\partial w}{\partial \nu}u\,d\mathcal H^{N-1},
\end{equation}
where $\nu$ denotes the outer unit normal exiting from $B_{\kappa\rho}$ and thanks to the fact that $w=0$ on $\partial B_{\kappa\rho}$ and $w\geq u$ on $\partial B_{\sqrt{\kappa}\rho}$.
We now observe that, since the torsion function on an annulus is explicit, with a direct computation one obtains\[
\left|\frac{\partial w}{\partial \nu}\right|\leq \beta_1\frac{\delta_\rho+\rho^2}{\rho},\qquad \text{on }\partial B_{\kappa\rho},
\]
for some $\beta_1=\beta_1(N,\kappa)$.
We can now combine~\eqref{eq:4.21} and~\eqref{eq:4.22} to obtain
\begin{equation}\label{eq:4.23}
\frac12\int_{B_{\kappa\rho}}|\nabla u|^2-\int_{B_{\kappa\rho}}u+\frac{\eta}{2}|\{u>0\}\cap B_{\kappa\rho}|\leq \beta_1(N,\kappa)\frac{\delta_\rho+\rho^2}{\rho}\int_{\partial B_{\kappa\rho}}u\,d\mathcal H^{N-1}.
\end{equation}
Then, using the definition of $\delta_\rho$, the trace inequality in $W^{1,1}$ and the arithmetic geometric mean inequality we obtain \[
\begin{split}
&\int_{\partial B_{\kappa\rho}}u\,d\mathcal H^{N-1}\leq C(N,\kappa)\left(\frac{1}{\rho}\int_{B_{\kappa\rho}}u+\int_{B_{\kappa\rho}}|\nabla u|\right)\\
&\leq \beta_2\left(\left(\frac{\delta_\rho}{\rho}+\frac12\right)|\{u>0\}\cap B_{\kappa\rho}|+\frac12\int_{B_{\kappa\rho}}|\nabla u|^2\right),
\end{split}
\]
for some $\beta_2=\beta_2(N,\kappa)>0$.
Putting together the above estimates, recalling again~\eqref{eq:4.20} we have, for all $\rho\leq \rho_0$ \[
\begin{split}
&\frac{\eta}{2}\int_{B_{\kappa\rho}}|\nabla u|^2+\frac{\eta}{2}|\{u>0\}\cap B_{\kappa\rho}|\\
&\leq \beta_1\frac{\delta_\rho+\rho^2}{\rho}\int_{\partial B_{\kappa\rho}}u\,d\mathcal H^{N-1}+\delta_\rho|\{u>0\}\cap B_{\kappa\rho}|\\
&\leq \beta_1(c(K_0+\rho)+\rho)\int_{\partial B_{\kappa\rho}}u\,d\mathcal H^{N-1}+c(K_0\rho+\rho^2)|\{u>0\}\cap B_{\kappa\rho}|\\
&\leq \beta_1\beta_2(c(K_0+\rho)+\rho)\left[ \left(\frac{\delta_\rho}{\rho}+\frac12\right)|\{u>0\}\cap B_{\kappa\rho}|+\frac12\int_{B_{\kappa\rho}}|\nabla u|^2\right]\\
&\hspace{50pt}+c(K_0\rho+\rho^2)|\{u>0\}\cap B_{\kappa\rho}|\\
&\leq \beta_1\beta_2(c(K_0+\rho)+\rho)\left(2c(K_0+\rho)+\frac12\right)\left[ \int_{B_{\kappa\rho}}|\nabla u|^2+|\{u>0\}\cap B_{\kappa\rho}|\right].
\end{split}
\]
Eventually, by choosing $K_0,\rho_0$ such that \[
\beta_1\beta_2(c(K_0+\rho_0)+\rho_0)\left(2c(K_0+\rho_0)+\frac12\right)\leq \eta/4,
\] 
we conclude that $u\equiv 0$ in $B_{\kappa\rho}$, for all $\rho\leq\rho_0$.
\end{proof}
%{\color{red} It is important to note that, if we decrease $\eta$, then one should decrease $m$ and $\rho_0$. (We have something like $m+\rho_0\leq \eta$)}

%
\begin{remark}\label{rmk:formulazioninondeg}
In  literature, the property proved in Lemma~\ref{le:lemma4.9} is called \emph{non-degeneracy}.
As it was noted for example in~\cite[Remark~2.8]{mtv}, there are two other equivalent versions of this result, where instead of the claim~\eqref{eq:nondegmean}, one can consider 
\begin{equation}\label{eq:nondegLinfty}
\|u\|_{L^\infty(B_\rho(x_0))}\leq K_0\rho\qquad \Longrightarrow \qquad u\equiv 0\text{ in }B_{\kappa\rho}(x_0)\cap B_R,
\end{equation}
or 
\begin{equation}\label{eq:nondegmeanpallapiena}
\mean{B_\rho(x_0)}{u\,dx}\leq K_0\rho\qquad \Longrightarrow \qquad u\equiv 0\text{ in }B_{\kappa\rho}(x_0)\cap B_R,
\end{equation}
 up to possibly modify the constants $K_0,\rho_0$ (but not their dependence only on $N,\kappa,\eta$).
\end{remark}

\begin{remark}\label{rmk:subsolutions}
As it was first highlighted in~\cite{bu}, Lemma~\ref{le:lemma4.9} holds for all sets that are optimal for a torsion energy-type functional only with respect to inward perturbations.
These sets are referred to as \emph{shape subsolutions} or \emph{inward minimizing sets} and one can easily prove that if $\Om$ is optimal for problem~\eqref{eq:minGeta1}, then it is a shape subsolution for the torsion energy.
Thus the non-degeneracy property of Lemma~\ref{le:lemma4.9} follows from~\cite[Theorem~2.2]{bu}.
Nevertheless we do not follow this approach  since for our scope we need finer regularity properties of optimal sets that can not be deduced only by means of inward perturbations.
\end{remark}
\begin{remark}
To obtain the regularity properties for minimizers we seek in this section, the previous lemma has to be paired with Lemma \ref{le:lemma4.10} below. Its proof is, as for the previous lemma, inspired by ~\cite[Lemma~4.10]{brdeve}, which is in turn based on \cite{alca}. One not completely obvious difference is that, contrary to the setting of~\cite{brdeve}, the parameter $\eta$ is not fixed in our setting, thus we need to keep track of it in the proofs. This dependence on $\eta$ will involve a dependence on $R$,  the radius of the ball containing all competitors in Theorem \ref{thm:mainvero2}. In particular the density estimates which ensue by the previous lemmata will depend on $R$, and this is a main obstacle in order to remove the equiboundedness hypothesis on competitors in \eqref{eq:mintor}.
\end{remark}

\begin{lemma}\label{le:lemma4.10}
Let $\alpha\in (0,N)$, $R$, $\eta$, $\eps$, $\Om$ and $u$ be as in Lemma~\ref{le:lemma4.9}. There exists a constant $M$, depending only on $N$, $\alpha$, $R$ and $\eta$ such that, for all $x_0\in B_R$, if 
\begin{equation}\label{eq:4.24}
\mean{\partial B_\rho(x_0)\cap B_R}{u\,d\mathcal H^{N-1}}\geq M\rho,
\end{equation}
then $u>0$ in $B_\rho(x_0)\cap B_R$.
\end{lemma}
\begin{proof}
First of all, we can reduce to the case when $B_\rho(x_0)\subset B_R$, up to take $M$ (depending only on $N,R$) big enough.
We define $v\in H^1_0(B_R)$ as the solution to 
\[
\begin{cases}
-\Delta v=1,\qquad \text{on }B_\rho,\\
v=u,\qquad \text{in }\R^N\setminus B_\rho(x_0).
\end{cases}
\]
By maximum principle we have $v>0$ in $B_\rho(x_0)$ and therefore \[
\{u>0\}\Delta \{v>0\}=\{u=0\}\cap B_\rho(x_0).
\]
Using this information, the quasi-minimality condition~\eqref{eq:uquasiminimizer} of $u$ and the property of the function $f_\eta$, see~\eqref{eq:propfeta}, we obtain\[
\frac12\int_{B_\rho(x_0)}|\nabla u|^2-\int_{B_\rho(x_0)}u\leq \frac12 \int_{B_\rho(x_0)}|\nabla v|^2-\int_{B_\rho(x_0)}v+\left(\frac{1}{\eta}+C\eps\right)|\{u=0\}\cap B_\rho(x_0)|,
\]
for some constant $C=C(N,\alpha,R)$.
Now we can use the equation satisfied by $v$ and the fact that $\eps<1<1/\eta$, to show\[
\frac12\int_{B_\rho(x_0)}|\nabla u-\nabla v|^2\leq \frac{C+1}{\eta}|\{u=0\}\cap B_\rho(x_0)|.
\]
Then, as in~\cite[Proof of Lemma~4.10]{brdeve} or in~\cite[Proof of Lemma~3.2]{alca}, one obtains\[
\frac{M^2}{2}|\{u=0\}\cap B_\rho(x_0)|\leq \frac{C+1}{\eta}|\{u=0\}\cap B_\rho(x_0)|,
\]
which by choosing $M\geq 2\sqrt{\frac{C+1}{\eta}}$ entails  that $|\{u=0\}\cap B_\rho(x_0)|=0$, and the proof is concluded.
\end{proof}

The main consequence of Lemmas~\ref{le:lemma4.9} and~\ref{le:lemma4.10} is the following result, stated first in~\cite[Section~3]{alca}, see also~\cite[Section~3 and~5]{velectures}.
\begin{lemma}\label{le:lemma4.11}
Let $\alpha\in (0,N)$, $R$, $\eta$, $\eps$, $\Om$ and $u$ be as in Lemma~\ref{le:lemma4.9}.
There exist constants $ \theta(N,\alpha,R,\eta)$ and $\rho_0(N,\alpha,R,\eta)$ such that 
\begin{enumerate}
\item[i)] $u$ is Lipschitz continuous with constant $L=L(N,\alpha,R)$. In particular, $\Om=\{u>0\}$ is an open set.
\item[ii)] For every $x_0\in \partial \Om$ and every $\rho\leq \rho_0$, we have 
{
\begin{equation}\label{eq:4.30}
\theta\leq \frac{|\Om\cap B_\rho(x_0)|}{|B_\rho|}\leq 1-\theta.
\end{equation}
}
\end{enumerate}
\end{lemma}

\begin{remark}
Notice that the constants determining the Lipschitz regularity and the density estimates of the previous result do not depend on $\eps$.
\end{remark}
This last result is the starting point of the higher regularity we need, that we treat in Section \ref{HR}.

%%%%%%%%%%%%%%%%%%%%%%%%%%%
%SECTION4: EQUIVALENCE%%%%%%%%%%%
%%%%%%%%%%%%%%%%%%%%%%%%%%%

%\newpage
\section{{Equivalence between the constrained and the unconstrained problem}}\label{sect:equivalence}
In this section we show that unconstrained minima of $\G$ and volume constrained minima of $\mathcal F_{\alpha,\eps}$ are actually the same. We begin by showing that
for $\eps$ small, the minimizers of $\mathcal G_{\eps,\eta}$ in $B_R$ are close to a ball in $L^\infty$.
To do that, we first start with an estimate that assures the $L^1-$proximity of an optimal set for problem~\eqref{eq:minGeta1} to a ball with radius not too large.
\begin{lemma}\label{le:stimadiffsimm}
Let $\alpha\in(0,N)$, $R>\omega_N^{-\frac1N}$ and $\eps,\eta\in(0,1)$. Let $\Om_{{\eps,\eta}}$ be an optimal set for~\eqref{eq:minGeta1} and $B_{{\eps,\eta}}$ a ball of measure $|\Omega_{\eps,\eta}|$ such that
\[
\mathcal A(\Om_{\eps,\eta})=\frac{|\Om_{{\eps,\eta}}\Delta B_{{\eps,\eta}}|}{|\Om_{\eps,\eta}|}.
\]
Then we have
\begin{equation}\label{eq:bounddiffsimm}
{|\Om_{{\eps,\eta}}\Delta B_{{\eps,\eta}}|\leq \frac{2C_0}{\sigma}|\Om_{{\eps,\eta}}|^{1+\frac{\alpha-2}{N}}\eps,}
\end{equation}
where $C_0(N,\alpha)>0$ is the constant appearing in Lemma~\ref{le:knupfermuratov} and $\sigma=\sigma(N)>0$ is the geometric constant from the quantitative Saint Venant inequality, see~\eqref{eq:quantsv}.
\end{lemma}
\begin{proof}
%Without loss of generality we suppose that $B_{{\eps,\eta}}$ is centered at the origin.
Using Lemma~\ref{le:knupfermuratov} and the definition of $f_\eta$, we get \[
\begin{split}
&E(\Om_{{\eps,\eta}})-E(B_{{\eps,\eta}})\leq \eps ( V_\alpha(B_{{\eps,\eta}})- V_\alpha(\Om_{{\eps,\eta}}))+(f_\eta(|B_{{\eps,\eta}}|)-f_\eta(|\Omega_{{\eps,\eta}}|))\\
&\leq C_0\eps|B_{{\eps,\eta}}\Delta\Om_{{\eps,\eta}}|\,\Big[|\Om_{\eps,\eta}|^{\frac{\alpha}{N}}+|B_{\eps,\eta}|^{\frac{\alpha}{N}}\Big].
\end{split}
\]
On the other hand, thanks to the quantitative version of the Saint-Venant inequality (Theorem~\ref{thm:quantitative}), and since $|\Om_{\eps,\eta}|=|B_{\eps,\eta}|$, we have (up to translations) that \[
\begin{split}
\sigma\left(\frac{|\Om_{{\eps,\eta}}\Delta B_{{\eps,\eta}}|}{|\Om_{{\eps,\eta}}|}\right)^2&\leq E(\Om_{{\eps,\eta}})|\Om_{{\eps,\eta}}|^{-1-\frac{2}{N}}-E(B_{{\eps,\eta}})|B_{{\eps,\eta}}|^{-1-\frac{2}{N}}\\
&\leq 2C_0\eps|\Om_{{\eps,\eta}}|^{-1+\frac{\alpha-2}{N}}|\Om_{{\eps,\eta}}\Delta B_{{\eps,\eta}}|
\end{split}
\]
so that
 \[
{|\Om_{{\eps,\eta}}\Delta B_{{\eps,\eta}}|\leq \frac{2C_0}{\sigma}|\Om_{{\eps,\eta}}|^{1+\frac{\alpha-2}{N}}\eps,}
\]
which proves the lemma.
\end{proof}
A consequence almost immediate of the previous lemma is that  the measure of the ball $B_{{\eps,\eta}}$ is not  too large.
\begin{lemma}\label{le:L1closetoball}
Let $\alpha$ and  $R$ be as in the previous lemma. There exists $\eta_1=\eta_1(N,\alpha,R)\leq \eta_0$ such that for all $\eps\in(0,1)$ and $\eta\leq \eta_1$, we have that any optimal set for problem~\eqref{eq:minGeta1} satisfies
 \[
|\Om_{{\eps,\eta}}|\leq 2.
\]
\end{lemma}
\begin{proof}
Of course the statement of the lemma is trivial as long as $|B_R|\le2$. Thus we take $R$ large enough so that $|B_R|>2$. 
Let us suppose for the sake of contradiction that $|\Om_{\eps,\eta}|>2$. We are then going to reach a contradiction as long as\[
1/\eta\ge C_a(N,\alpha)R^{N+2}+C_b(N,\alpha),
\] 
for  given constants $C_a(N,\alpha)$ and $C_b(N,\alpha)$ which will be precised later on in the proof. Since the functional\[
\eps\mapsto \mathcal G_{\eps,\eta}(\Om_{{\eps,\eta}}),
\]
is nondecreasing, we get  \[
\sup_{\eps\in(0,1)}\mathcal G_{\eps,\eta}(\Om_{{\eps,\eta}})=\mathcal G_{1,\eta}(\Om_{1,\eta})\leq E(B)+ V_\alpha(B),
\]
where $B$ is a ball of unit measure.
On the other hand, using the Saint-Venant inequality, the positivity of $ V_\alpha$, the fact that $\Om_{{\eps,\eta}}\subset B_R$ and since $|\Om_{{\eps,\eta}}|>2$ we have
\[
E(B)+ V_\alpha(B)\geq \mathcal G_{\eps,\eta}(\Om_{{\eps,\eta}})\geq E(B_{{\eps,\eta}})+\frac{1}{\eta}(|\Om_{{\eps,\eta}}|-1)\geq \omega_N^{\frac{N+2}{N}} R^{N+2}E(B)+\frac{1}{\eta}.
\]
By letting $C_a(N,\alpha)= (-E(B))\omega_N^{\frac{N+2}{N}}$ and $C_b(N,\alpha)=E(B)+ V_\alpha(B)$, and by choosing $\eta_1=\eta_1(N,\alpha,R)$ such that  $\eta_1\leq \eta_0$ and
\[
\frac{1}{\eta_1}>E(B)+ V_\alpha(B)+(-E(B))\omega_N^{\frac{N+2}{N}} R^{N+2},
\]
we reach the desired contradiction.
\end{proof}
We note that in the above lemma, $\eta_1$ depends on $R$ and in particular $\eta_1\approx \frac{1}{R^{N+2}}$.

\begin{corollary}\label{cor:vicinanzaL1}
In the assumptions of Lemma~\ref{le:stimadiffsimm}, there exists a positive constant $c_1=c_1(N,\alpha)$ such that,  for all $\eps\in(0,1)$ and $\eta\leq \eta_1$, we have 
\begin{equation}\label{eq:stimaL1}
|\Om_{{\eps,\eta}}|\leq 2,\qquad |\Om_{{\eps,\eta}}\Delta B_{{\eps,\eta}}|\leq c_1\eps.
\end{equation}  
\end{corollary}
\begin{proof}
It is a direct consequence of Lemmas~\ref{le:L1closetoball} and  \ref{le:stimadiffsimm}.
\end{proof}
Next we show that, for $\eps$ small, the boundary of any  optimizer $\Om_{{\eps,\eta}}$ is close to the one of the corresponding optimal ball $B_{{\eps,\eta}}$ in the definition of asymmetry, with respect to the Hausdorff distance $d_H$ (see \cite[Definition 4.4.9]{amti} for the definition and properties of the Hausdorff distance).
\begin{lemma}\label{le:closehausdorff}
Under the assumptions of Corollary~\ref{cor:vicinanzaL1}, for all $\delta>0$ there exists $\eps_\delta=\eps_\delta(\delta, N, \alpha, R)\in(0,\eps_0)$ such that for all $\eps\leq \eps_\delta$, we have \[
{\rm dist}_H(\partial \Om_{{\eps,\eta}},\partial B_{{\eps,\eta}})\leq \delta.
\] 
\end{lemma}
\begin{proof}
By~\eqref{eq:stimaL1} we have that $|\Om_{{\eps,\eta}}\setminus B_{{\eps,\eta}}|\leq c_1\eps$.
We fix $\delta>0$ and call $B_\delta(B_{{\eps,\eta}}):=B_{{\eps,\eta}}+B_\delta$ the $\delta$-neighborhood of $B_{{\eps,\eta}}$.
If $\Om_{{\eps,\eta}}\setminus B_\delta(B_{{\eps,\eta}})$ is empty, then there is nothing to prove. Otherwise there exists $x\in \Om_{{\eps,\eta}}\setminus B_\delta(B_{{\eps,\eta}})$ so that by point $(ii)$ of Lemma~\ref{le:lemma4.11} there exists $\rho_0(N,R,\alpha)$ such that for   $\rho\leq \rho_1:=\min\{\rho_0(N,R,\alpha),\delta\}$ it holds
 \[
{\omega_N \theta}\rho^N\leq |B_\rho(x)\cap \Om_{{\eps,\eta}}|\leq |\Om_{{\eps,\eta}}\setminus B_{{\eps,\eta}}|\leq c_1\eps.
\] 
Notice that the choice of $\rho_1\leq \delta$ assures that $|B_\rho(x)\cap \Om_{{\eps,\eta}}|\leq |\Om_{{\eps,\eta}}\setminus B_{{\eps,\eta}}|$. In conclusion choosing $\rho=\rho_1$, we have \[
{\omega_N\theta}\rho_1^N\leq c_1\eps,
\]
which is not possible as soon as \[
\eps\leq \eps_\delta:=\frac{{\omega_N\theta}}{c_1}\rho_1^N.
\]
With the same argument, by using the density estimates for the exterior of $\Om_{\eps,\eta}$, we show that  $B_{{\eps,\eta}}\subset B_\delta(\Om_{{\eps,\eta}})$ where $B_\delta(\Om_{{\eps,\eta}}):=\Om_{{\eps,\eta}}+B_\delta$. This concludes the proof.
\end{proof}
It is worth noting that the constant $\eps_\delta$ in the lemma above depends also on $R$. This is one of the main difficulties in trying to get rid of the equiboundedness assumption of Theorem~\ref{thm:mainvero2}.

\begin{remark}\label{rmk:closehausdorff}
In view of the {previous} result, we fix $\eps_1(N,\alpha, R)$ as the $\eps_\delta$ from Lemma~\ref{le:closehausdorff} with the choice of $\delta:=1/2$.

If $\eps\leq \eps_1$, then in the proof of Theorem~\ref{thm:noconstraint}, we will be allowed to inflate a set while remaining in a sufficiently big ball $B_R$.
\end{remark}

We can now show the equivalence between the constrained and the unconstrained problems. We will use the following elementary lemma.
\begin{lemma}\label{yawn}
Let $\alpha\in[0,N]$, $P,Q>0$ be two positive real numbers and let $u:[0,1)\to\R$ be the function defined by
\[
u(t)=\frac{P(1-t^{N+2})-Q(1-t^{N+\alpha})}{1-t^N}.
\]
Then there exists $q=q(N,\alpha,P)>0$ and $C=C(N,\alpha,P)>0$ such that $\inf_{[0,1)}u \ge C(N,\alpha)$ for any $Q<q$. 
\end{lemma}
\begin{proof}
Let us write $u(t)=Pf(t)-Qg(t)$ where
\[
f(t)=\frac{1-t^{N+2}}{1-t^N}\qquad g(t)=\frac{1-t^{N+\alpha}}{1-t^N}.
\]
Both $f$ and $g$ can be extended by continuity in $1$ with the values, respectively, of $f(1)=\frac{N+2}{N}$ and $g(1)=\frac{N+\alpha}{N}$. Since such extensions are continuous and strictly positive on $[0,1]$, they admit strictly positive minimum and maximum in there. Let $m_f=\min_{[0,1]} f>0$ and $M_g=\max_{[0,1]}g$. Then we get, for any $t\in[0,1)$, that
\[
u(t)=Pf(t)-Qg(t) \ge Pm_f -Q M_g.
\]
We conclude the proof by observing that, as long as 
\[
Q<\frac{Pm_f}{2M_g}:=q,
\]
we have, for all $t\in[0,1)$, \[
u(t)\geq Pm_f -Q M_g\geq Pm_f -q M_g=\frac{P m_f}{2}=:C(N,\alpha,P),
\]
and the claim is proved.
\end{proof}
%}

\begin{theorem}\label{thm:noconstraint}
Let $\alpha\in (0,N)$ and $B$ be a ball of unit measure. There exists $R_0=R_0(N)$ such that, for all $R\geq R_0$, there exists $ \eps_2= \eps_2(N,\alpha,R)\leq \eps_1$ and $\eta_2= \eta_2(N,\alpha,R)\leq \eta_1$ such that, for all $\eta\leq \eta_2$ and $\eps\leq \eps_2$, we have that 
\begin{equation}\label{eq:nocontraint}
\begin{aligned}
&\min\left\{\mathcal G_{\eps,\eta}(\Om) : \Om\subset B_R \right\}
\\
&\ge\inf\left\{\mathcal F_{\alpha,\eps}(\Om) : \Om\subset B_R,\;  |\Om|=1\right\}=:\mu(N,\alpha,\eps, R). 
\end{aligned}
\end{equation}
As a consequence, problems \eqref{eq:main} and \eqref{eq:minGeta1} are equivalent.
\end{theorem}
\begin{proof}
It is easy to check that \[
\begin{split}
\min\left\{\mathcal G_{\eps,\eta}(\Om) : \Om\subset B_R\right\}
\le\inf\left\{\mathcal F_{\alpha,\eps}(\Om) : \Om\subset B_R,\;  |\Om|=1\right\},
\end{split}
\]
as the two functionals coincide on sets of measure $1$. 
Then, if the first claim of the theorem holds, it follows that on the set of minimizers (of the first or of the second problem) the two functionals do coincide, that is, problems \eqref{eq:main} and \eqref{eq:minGeta1} are equivalent.

We prove the first claim of the theorem  by contradiction. Let 
 \[
\Om_{{\eps,\eta}}\subset B_R,\quad \sigma_{{\eps,\eta}}\in\R,\quad |\Om_{{\eps,\eta}}|=1+\sigma_{{\eps,\eta}},\quad \mathcal G_{\eps,\eta}(\Om_{{\eps,\eta}})<\mu,
\]
and we also note that, since for all $\Om\subset B_R$ it holds $\mathcal F_{\alpha,\eps}(\Om)\leq \eps V_\alpha(B),$ then $\mu\leq \eps V_\alpha(B)$.
We moreover assume, without loss of generality, that $\Om_{{\eps,\eta}}$ are minimizers for problem~\eqref{eq:minGeta1}.
We treat separately the case $\sigma_{{\eps,\eta}}>0$ and $\sigma_{{\eps,\eta}}<0$.

\par
{\bf Case $\sigma_{{\eps,\eta}}>0$. }
We first observe that $\sigma_{{\eps,\eta}}\to0$ as $\eta\to0$. Indeed
\[
\G(\Om_{{\eps,\eta}})=\mathcal F_{\alpha,\eps}(\Om_{{\eps,\eta}})+\frac{1}{\eta}\sigma_{{\eps,\eta}}
\]
and so 
\[
0\le\frac{1}{\eta}\sigma_{{\eps,\eta}}= \mathcal G_{\eps\eta}(\Om_{\eps,\eta})-\mathcal F_{\alpha,\eps}(\Omega_{{\eps,\eta}})\le \eps V_\alpha(B) -E(B_R),
\]
using the assumption $\mathcal G_{\eps,\eta}(\Om_{\eps,\eta})\leq \mu\leq \eps V_\alpha(B)$, the positivity of $ V_\alpha$ and the fact that the torsion energy is decreasing by inclusion. This implies that $\sigma_{{\eps,\eta}}\to0$ as $\eta\to0$.

Let now $\lambda_{{\eps,\eta}}<1$ be such that $|\lambda_{{\eps,\eta}}\Om_{{\eps,\eta}}|=1$, therefore 
\[
 \lambda_{{\eps,\eta}}=1-\sigma_{{\eps,\eta}}\frac{1}{N(1+\sigma_{{\eps,\eta}})}+C\sigma_{{\eps,\eta}}^2,
\]
for some  $C=C(N)\in \R$.
Since the new set $\lambda_{{\eps,\eta}}\Om_{{\eps,\eta}}$ is now admissible in the constrained minimization problem~\eqref{eq:main}, and since
 \[
\begin{split}
&\G(\Om_{\eps,\eta})=E(\Om_{{\eps,\eta}})+\eps  V_\alpha(\Om_{{\eps,\eta}})+\frac{\sigma_{{\eps,\eta}}}{\eta}\\
&<\mu\leq E(\Om_{{\eps,\eta}})\lambda_{{\eps,\eta}}^{N+2}+\eps  V_\alpha(\Om_{{\eps,\eta}})\lambda_{{\eps,\eta}}^{N+\alpha}\\
&=E(\Om_{{\eps,\eta}})\left(1-\sigma_{{\eps,\eta}}\frac{N+2}{N(1+\sigma_{{\eps,\eta}})}+C\sigma_{{\eps,\eta}}^2\right)+\eps  V_\alpha(\Om_{{\eps,\eta}})\left(1-\sigma_{{\eps,\eta}}\frac{N+\alpha}{N(1+\sigma_{{\eps,\eta}})}+C\sigma_{{\eps,\eta}}^2\right),
\end{split}
\]
we deduce that
\[
\begin{aligned}
\frac{\sigma_{{\eps,\eta}}}{\eta}&<(-E(\Om_{{\eps,\eta}}))\sigma_{{\eps,\eta}}\frac{N+2}{N(1+\sigma_{{\eps,\eta}})}-\eps  V_\alpha(\Om_{{\eps,\eta}})\sigma_{{\eps,\eta}}\frac{N+\alpha}{N(1+\sigma_{{\eps,\eta}})}+C\sigma_{{\eps,\eta}}^2\mathcal F_{\alpha,\eps}(\Om_{{\eps,\eta}})\\
&\le (-E(\Om_{{\eps,\eta}}))\sigma_{{\eps,\eta}}\frac{N+2}{N(1+\sigma_{{\eps,\eta}})}+C\sigma_{{\eps,\eta}}^2\eps V_\alpha(B),
\end{aligned}
\]
where we have again used the fact that $\eps V_\alpha(B)$ bounds from above the functional $\mathcal F_{\alpha,\eps}$.
Thus 
\[
\frac{1}{\eta}\le C(N,\alpha)(-E(\Om_{{\eps,\eta}})) \le -C(N,
\alpha)E(B_R),
\]
which leads to a contradiction as soon as $\eta_2<\frac{1}{C(N,\alpha)(-E(B_R))}=\frac{C(N,\alpha)}{R^{N+2}}$.

\par
{\bf Case $\sigma_{{\eps,\eta}}<0$. } For this case let us call \[
\rho_{{\eps,\eta}}:=(1+\sigma_{{\eps,\eta}})^{-1/N},
\]
so that $|\rho_{{\eps,\eta}}\Om_{{\eps,\eta}}|=1$.

We recall from the previous sections that a minimizer $\Om_{{\eps,\eta}}$ for $\G$ exists, and by Lemma~\ref{le:closehausdorff}, up to take $\eps_2\leq \eps_1$ as in Remark~\ref{rmk:closehausdorff}, and $\eta_2<\eta_1$ as in Lemma~\ref{le:L1closetoball}, the rescaled set $\rho_{{\eps,\eta}}\Om_{{\eps,\eta}}$ is still contained in $B_R$, as soon as, for example, $R_0>6$.

In fact, thanks to Lemma~\ref{le:infneg} and the Saint Venant inequality, we have
\begin{equation}\label{eq:boundrho}
E(B)\leq E(\rho_{\eps,\eta}\Om_{\eps,\eta})\leq \rho_{\eps,\eta}^{N+2}\frac{E(B)}{4},\qquad \text{hence, }\qquad \rho_{\eps,\eta}^{N+2}\leq 4,
\end{equation}
therefore, it is easy to check that $\rho_{\eps,\eta}\Om_{\eps,\eta}\subset B_R$.

Let us define the function\[
g\colon[1,\rho_{{\eps,\eta}}]\rightarrow \R,\qquad g(r)=E(r\Om_{\eps,\eta})+\eps  V_\alpha(r\Om_{{\eps,\eta}})+\eta(r^N|\Om_{{\eps,\eta}}|-1).
\]
We want to show that the minimum of the function $g$ is attained at $r=\rho:=\rho_{{\eps,\eta}}$. This is equivalent to show that for some $\eta$ the inequality  
\[
g(r)\geq E(\rho\Om_{{\eps,\eta}})+\eps  V_\alpha(\rho\Om_{{\eps,\eta}}),\qquad \text{for all }r\in[1,\rho],
\]
holds true. Up to rearranging the terms, and by the homogeneity of the functionals $E$ and $ V_\alpha$ such an inequality reads as
\[
\eta\left(1-\left(\frac{r}{\rho}\right)^N\right)\le (-E(\rho\Om_{{\eps,\eta}}))\left(1-\left(\frac{r}{\rho}\right)^{N+2}\right)-\eps V_\alpha({\rho}\Om_{\eps,\eta})\left(1-\left(\frac{r}{\rho}\right)^{N+\alpha}\right).
\]
Setting $t:=\frac r\rho < 1$, and observing that $r^N|\Om_{{\eps,\eta}}|=t^N$, the last inequality is equivalent to
\[
\eta\le \frac{(-E(\rho\Om_{{\eps,\eta}}))(1-t^{N+2})-\eps  V_\alpha(\rho\Om_{{\eps,\eta}})(1-t^{N+\alpha})}{1-t^N}.
\]
We recall now that $ V_\alpha(\rho\Om_{{\eps,\eta}})\le  V_\alpha(B)$ by the { Riesz} inequality, while $E(\rho\Om_{{\eps,\eta}})\le \rho^{N+2} \frac{E(B)}{4}\leq E(B)$, by Lemma~\ref{le:infneg} and~\eqref{eq:boundrho}. Thus   
\[
\begin{split}
\frac{(-E(\rho\Om_{{\eps,\eta}}))(1-t^{N+2})-\eps  V_\alpha(\rho\Om_{{\eps,\eta}})(1-t^{N+\alpha})}{1-t^N}\\
\ge \frac{-E(B)(1-t^{N+2})-\eps V_\alpha(B)(1-t^{N+\alpha})}{1-t^N}.
\end{split}
\]
Thus it is enough to show that for some $\eta>0$ it holds
\[
\eta\le \frac{-E(B)(1-t^{N+2})-\eps V_\alpha(B)(1-t^{N+\alpha})}{1-t^N}:=u_\eps(t).
\]
To conclude that $u_\eps>0$ in $[0,1)$ we directly apply Lemma \ref{yawn} with $u_\eps$ in place of $u$,  $-E(B)$ in place of $P$, and $\eps V_\alpha(B)$ in place of $Q$. Up to choose $\eps_2$ small enough, depending only on $N$ and $\alpha$, we can satisfy the requirement of the Lemma. This concludes the proof.
\end{proof}

We highlight that, from now on, we can fix an $\eta>0$ so that Theorem~\ref{thm:noconstraint} holds true, and therefore we have the equivalence of the constrained minimization problem for $\mathcal F_{\alpha,\eps}$ and the unconstrained problem for $\G$.
It is then consistent to denote an optimal set for $\G$ or $\mathcal F_{\alpha,\eps}$ by $\Om_\eps$ (and $u_\eps$ its torsion function), dropping the dependence on $\eta$.

On the other hand, we stress that this choice of $\eta$ does depend on $R$!

%%%%%%%%%%%%%%%%%%%%%%%%%%%
%SECTION5: Higher Regularity%%%%%%%%%%%
%%%%%%%%%%%%%%%%%%%%%%%%%%%

%\newpage
\section{Higher regularity of minimizers}\label{HR}

In this section we show that the mild regularity proved in Section \ref{MR} can be improved to a higher regularity of minimizers for $\G$ or, equivalently, $\mathcal F_{\alpha,\eps}$. More precisely, we will show that minimizers of $\mathcal F_{\alpha,\eps}$ are such that their boundary can be parametrized on the sphere so that the  $C^{2,\gamma}-$norm of such a  perturbation is arbitrarily small, up to choose $\eps$ small enough. 

For this whole section, we fix $R>R_0$ and $\eps\leq \eps_2(N,\alpha, R)$ so that Theorem~\ref{thm:noconstraint} holds. Then we denote $\Om_\eps$ an optimal set for problem~\eqref{eq:minGeta1} and let $u_\eps$ be its torsion function, extended to zero outside $\Om_\eps$. Hence $u_\eps$ is optimal for problem~\eqref{eq:minpropu}.

We begin with a simple geometric result, whose proof is just a rephrasing of Lemma \ref{le:closehausdorff}, since now we have the additional information that $|\Om_\eps|=1$.
\begin{lemma}\label{le:lemma5.1}
With the notations above, the sequence $\Omega_\eps$ converges to $B$ in $L^1$ as $\eps\to 0$. Moreover, for any $\delta>0$ there exists $\eps_\delta>0$ such that if $\eps<\eps_\delta$, then 
\[
\partial \Omega_\eps\subset \partial B + B_\delta = \{x\in\R^N\, :\, {\rm dist}(x,\partial B) <\delta \}. 
\]
\end{lemma}
To get the desired regularity of minimizers, we will apply results from \cite{alca}, and  techniques developed in \cite{brdeve}, and later on in \cite{demamu}.

We will need the following result~\cite[Theorem~4.5 and Theorem~4.8]{alca}, \cite[Theorem~2]{agalca}.
\begin{theorem}\label{thm:agalcathm2}
Let $\eps\leq \eps_2$, $\Om_\eps$ and $u_\eps$ be as above. The following facts hold true.
\begin{enumerate}
\item[(i)] There is a Borel function $q_{u_\eps}\colon \partial \Om_\eps\rightarrow \R$ such that, in the sense of the distributions, one has
\begin{equation}\label{eq:bdp4.31}
-\Delta u_\eps = \chi_{\Omega_\eps} - q_{u_\eps}\mathcal H^{N-1}\lfloor\partial\Omega_\eps,\qquad\text{ in }B_R.
\end{equation}
\item[(ii)]  There exist constants $0<c<C<+\infty$, depending on $R$, $N$, $\alpha$,  such that $c\le q_{u_\eps}\le C$.
\item[(iii)] For all points $\overline x\in \partial^*\Om_\eps=\partial^*\{u_\eps>0\}$, the measure theoretic {inner} unit normal $\nu_{u_\eps}(\overline x)$ is well defined and, as $\rho\to0$, 
\begin{equation}\label{eq:bdp4.32}
\frac{\Om_\eps-\overline x}{\rho}\rightarrow \{x : x\cdot \nu_{u_\eps}(\overline x)\geq 0\},\qquad \text{in }L^1(B_R).
\end{equation}
\item[(iv)]For $\mathcal H^{N-1}$ almost all $\overline x\in \partial^*\{u_\eps>0\}$ we have 
\begin{equation}\label{eq:bdp4.33}
\frac{u_\eps(\overline x+\rho x)}{\rho}\longrightarrow q_{u_\eps}(\overline x)(x\cdot \nu_{u_\eps}(\overline x))_+,\qquad \text{in }W^{1,p}(B_R)\;\text{for every }p\in[1,+\infty).
\end{equation}
\item[(v)] $\mathcal H^{N-1}(\partial \Om_\eps\setminus\partial^*\Om_\eps)=0$.
\end{enumerate}
\end{theorem}

\begin{remark}[On the meaning of $q_{u_\eps}$]
For a regular set $\Omega$, by means of a shape derivative argument, one can show that  $q_{u_\eps}(x)=|{\partial_\nu u_\eps}|(x)$ for $x\in\partial \Omega_\eps=\partial\{u_\eps>0\}$. The slightly more complicated arguments that follow are due since we only know, for the moment, that minimizers of problem \eqref{eq:main} are open sets of finite perimeter.  
Namely, following ideas from \cite{alca} and \cite{agalca}, in order to show some higher regularity we first need to show some regularity results for $q_{u_\eps}$. Formally it is possible to see that the first variation of $\mathcal G_{\eps,\eta}$ reads as

\[
\left|\frac{\partial u_\eps(x)}{\partial \nu}\right|^2 {-\eps}v_{\Om_\eps}(x)=\Lambda,\qquad x\in \partial\Om_\eps,
\]
where $u_\eps$ is the torsion function of $\Omega_\eps$, $v_{\Om_\eps}$ its Riesz potential and  $\Lambda$ some constant. Thus, since $q_{u_\eps}$ stays far from zero and infinity (thanks to Theorem~\ref{thm:agalcathm2}(ii)), then the regularity of $q_{u_\eps}=\left|\frac{\partial u_\eps}{\partial \nu}\right|$ is the same as that of $v_{\Om_\eps}$. Such relation on the other hand is not necessarily true, because of the lack of regularity of $\partial\Omega$, but will turn out to be true on $\partial^*\Omega$, the reduced boundary of $\Omega$.
\end{remark}

Before rigorously developing the argument described in the previous remark, we show a simple regularity result for the Riesz potentials. This is rather standard, but we give a proof for the sake of completeness.
\begin{lemma}\label{le:rieszregularity}
Let $\alpha\in (1,N)$ and let $A$ be a bounded open set. Then $w:=\chi_A*|\cdot|^{\alpha-N}$ is of class $C^{1,\gamma}(\overline A)$ for some $\gamma\in(0,1)$. 
\end{lemma}
\begin{proof}
Let 
\[w_\eps(x)=\int_A\frac{dy}{(|x-y|+\varepsilon)^{N-\alpha}},\qquad \text{for }x\in A,
\] 
and 
\[
w_i(x)=\int_A\partial_{x_i}\left(\frac{1}{|x-y|^{N-\alpha}}\right)\,dy=(\alpha-N)\int_A\frac{x_i-y_i}{|x-y|^{N-\alpha+2}}\,dy,
\] 
for $i=1,\dots,N$. Notice that, where $|x-y|\approx 0$, then \[
\frac{x_i-y_i}{|x-y|^{N-\alpha+2}}\approx \frac{1}{|x-y|^{N-\alpha+1}}.
\] 
Since $\alpha>1$, then $N-\alpha+1<N$ so that the $w_i$ are well defined. It is also clear that $w_\eps$ is a smooth function. We define $A_\eps^1:=\{y\in A\,:\, |x-y|\ge\sqrt\eps\}$ and $A_\eps^2=A\setminus A_\eps^1$. Notice that by absolute continuity of the Lebesgue integral, it holds 
\[
\int_{A_\eps^2} \frac{x_i-y_i}{|x-y|^{N-\alpha+2}}\,dy=o_\eps(1)\quad \text{ and } \quad \int_{A_\eps^2} \partial_{x_i} \frac{dy}{(|x-y|+\varepsilon)^{N-\alpha}}=o_\eps(1),
\]
where $o_\eps(1)$ does \emph{not} depend on $x$, but only on the measure $|A_\eps^2|$.
Thanks to this, we have that (for a constant $C$ depending only on $N$, $\alpha$),
\[ 
\begin{aligned}
&|\partial_{x_i} w_\eps(x)-w_i(x)|\\
&=\left|(\alpha-N)\int_{A_\eps^1}\frac{x_i-y_i}{|x-y|^{N-\alpha+2}}  \left[ \frac{1}{\left(1+\frac{\eps}{|x-y|} \right)^{N-\alpha+2}}-1\right] \,dy \right| +o_\eps(1)\\
&\le (N-\alpha)(N-\alpha+2)\int_{A_\eps^1} \frac{1}{|x-y|^{N-\alpha+1}}\left( \frac{\eps}{|x-y|}+o\left(\frac{\eps}{|x-y|}\right) \right)+o_\eps(1)\\
&\le C\sqrt\eps+o_\eps(1)=o_\eps(1). 
\end{aligned}
\]
Thus $\partial_{x_i} w_\eps(x)-w_i(x)\to0$ uniformly in $\R^N$. Since $w_\eps$ converges pointwise to $w$, this implies that $w$ is derivable and that
\[
\partial_{x_i} w(x)=(\alpha-N)\int_A\frac{x_i-y_i}{|x-y|^{N-\alpha+1}}\,dy.
\]
It is now easy to show that $\partial_{x_i} w(x)$ is an H\"older continuous function. This concludes the proof.
\end{proof}
In what follows we drop the subscript $\eps$ from $\Om_\eps$ and $u_\eps$ as here $\eps$ is fixed and there is no risk of confusion. The general strategy, and part of the details in the proof of the following theorem  are inspired by an argument first proposed in~\cite{agalca} and readapted later on in \cite{brdeve}.
\begin{theorem}\label{thm:optcondqu}
Let $R>R_0$, $\alpha\in(1,N)$ and $\eps\leq \eps_2$, and let $\Omega$ be a minimizer for $\G$, $u$ be its torsion function, $v_\Omega=v=\chi_\Omega*|\cdot|^{\alpha-N}$ be its Riesz potential and $q_u$ be as in Theorem~\ref{thm:agalcathm2}. Then the function {$x\mapsto q_u^2(x)-\eps v(x)$} is constant on  $\partial^*\Omega$.
\end{theorem} 
\begin{proof}
Let us assume, for the sake of contradiction, that there are  $x_0,x_1\in\partial^*\Omega$ such that
\[
{q_u^2(x_0)-\eps v(x_0)< q^2_u(x_1)-\eps v(x_1).}
\] 
We construct a family of diffeomorphisms which preserves the volume at the first order by {deflating} $\Omega$ around $x_0$, and {inflating} it around $x_1$. Let $\kappa<1$ and $\rho<1$ be two parameters. Let $\varphi\in C^1_0(B_1(0))$ be a non-null, radially symmetric function supported in $B_1(0)$. Then we define, { keeping in mind that $\nu_{x_i}$ denotes the inner normal},
\[
\tau_{\rho,\kappa}(x)=\tau(x)=
x+\sum_{i\in\{0,1\}}(-1)^i\kappa\rho\varphi\left(\frac{|x-x_i|}{\rho}\right)\nu_{x_i} \chi_{B_\rho(x_i)}.
\]

The field $\tau$ is a diffeomorphism for $\rho$ and $\kappa$ small enough. Notice that $\tau(x)-x$ is null outside $B_\rho(x_0)\cup B_\rho(x_1)$. A simple computation shows that
\[
\nabla\tau(x)=Id+\sum_{i\in\{0,1\}}(-1)^i\kappa\varphi'\left(\frac{|x-x_i|}{\rho}\right)\frac{x-x_i}{|x-x_i|}\otimes\nu_{x_i} \chi_{B_\rho(x_i)}, 
\]
so that\footnote{We are using the formula $\det(Id+\xi B)=1+   trace(B)\xi+o(\xi)$.}
\begin{equation}\label{eq:espdet}
\det(\nabla\tau(x)) = 1+ \sum_{i\in\{0,1\}}(-1)^i\kappa\varphi'\left(\frac{|x-x_i|}{\rho}\right)\frac{x-x_i}{|x-x_i|}\cdot\nu_{x_i} \chi_{B_\rho(x_i)}+o(\kappa).
\end{equation}
We call $\Omega_\rho =\tau(\Omega)$. We are going to show that for  $\kappa,\rho$ small enough it holds $\G(\Omega_\rho)< \G(\Omega)$, contradicting the minimality of $\Omega$. To do that we deal with the first variation of each term of the sum defining $\G$. We stress that the computations regarding the volume and the torsion contributions are identical to those performed originally in \cite{agalca} (see also \cite{brdeve} and \cite{demamu}, where the same idea is applied). We add them for the sake of completeness.

Let us begin with the volume term. We claim that
\begin{equation}\label{variazionevolume}
f_\eta(\Omega_\rho)-f_\eta(\Omega)= o(\rho^N),\qquad \text{as }\rho\rightarrow 0.
\end{equation}
To see that, thanks to \eqref{eq:propfeta} we only have to show that 
\[
\frac{1}{\rho^N}(|\Omega_\rho|-|\Omega|)\to 0,\qquad \text{as }\rho\to0. 
\]
Using the Area formula and the change of variables $x=x_i+\rho y$, we have that
\[
\begin{aligned}
\frac{1}{\rho^N}(|\Omega_\rho|-|\Omega|)&=\frac{1}{\rho^N}\left(\int_{\Omega_\rho}1\,dx-\int_\Omega1\,dx\right)\\
&=\frac{1}{\rho^N}\sum_{i\in\{0,1\}}\Big(\int \chi_{\Omega_\rho\cap B_\rho(x_i)}(x)\,dx-\int\chi_{\Omega\cap B_\rho(x_i)}(x)\,dx\Big)\\
&=\sum_{i\in\{0,1\}} \int \chi_{\left(\frac{\Omega-x_i}{\rho}\right)\cap B_1(0)}(y) \det(\nabla\tau(x_i+\rho y)) - \chi_{\left(\frac{\Omega-x_i}{\rho}\right)\cap B_1(0)}(y)\,dy.
\end{aligned}
\]
{  
We can then deduce by Theorem~\ref{thm:agalcathm2} point (iii) that   $\frac{\Omega_\rho-x_i}{\rho} \to {\{x\cdot \nu_{x_i}\geq 0\}}$ in $L^1(B_R)$, whence
\[
\lim_{\rho\to 0}\frac{1}{\rho^N}(|\Omega_\rho|-|\Omega|)=\sum_{i\in\{0,1\}} \int_{\{x\cdot \nu_{x_i}>0\}\cap B_1(0)} (-1)^i\kappa\varphi'(|y|)\left(\frac{y}{|y|}\right)\cdot \nu_{x_i} \,dy =0,
\]
where the last equality is due to the radial symmetry of $\varphi$.
}
Now that \eqref{variazionevolume} is settled, we deal with the torsion energy term. We claim that
\begin{equation}\label{variazionetorsione}
\frac{1}{\rho^N} (E(\Omega_\rho)-E(\Omega))\le\kappa{(q_u(x_0)^2-q_u(x_1)^2)}C(\varphi) + o_\rho(1)+o(\kappa),
\end{equation}
where 
{
\begin{equation}\label{rfk}
C(\varphi)=\int_{B_1(0)\cap\{y\cdot\nu=0\}}\varphi(|y|)\,d\mathcal H^{N-1}(y)=-\int_{B_1(0)\cap \{y\cdot \nu>0\}}\varphi'(|y|)\frac{y\cdot \nu}{|y|}\,dy,
\end{equation}
and the last equality follows from the divergence Theorem, recalling that $\nu$ is a inner normal and ${\mathrm{div}}(\varphi(|y|)\nu)=\varphi'(|y|)\frac{y\cdot \nu}{|y|}$.
} 
Moreover, we note that $\nu$ can be any unit direction of $\R^N$: changing direction does not affect the value of $C(\varphi)$, thanks to the radial symmetry of $\varphi$.
To show~\eqref{variazionetorsione} it suffices to prove that 
\begin{equation}\label{t1}
\frac{1}{\rho^N} \left(\int_{\Omega_\rho}|\nabla \uu_\rho|^2\,dx-\int_{\Omega}|\nabla u|^2\,dx\right)={\kappa(q_u(x_0)^2-q_u(x_1)^2)C(\varphi)} + o_\rho(1)+o(\kappa),
\end{equation}
where $\uu_\rho=u\circ \tau^{-1}$, and that
\begin{equation}\label{t2}
\int_{\Omega_\rho}\uu_\rho\,dx - \int_{\Omega}u\,dx=o(\rho^N),\qquad\text{as }\rho\to0.
\end{equation}

Indeed $\uu_\rho$ is a test function in the definition of $E(\Omega_\rho)$ so that \eqref{t1} and \eqref{t2} imply directly \eqref{variazionetorsione}. 

{The computation of \eqref{t1} is exactly as in~\cite[Proof of Lemma~4.15]{brdeve} (it is done also in \cite[Section~2]{agalca} and \cite{demamu}), hence we do not repeat it here.} To show \eqref{t2}, we compute
\[
\begin{aligned}
\frac{1}{\rho^N}&\left(\int_{\Omega_\rho}\uu_\rho(z)\,dz-\int_{\Omega}u(x)\,dx\right)=\frac{1}{\rho^N}\int_\Omega \Big(u\circ\tau^{-1}(\tau(x))\det(\nabla\tau(x))-u(x)\Big)\,dx\\
&=\sum_{i\in\{0,1\}}\int_{B_1(0)\cap\left(\frac{\Omega-x_i}{\rho}\right)} \Big(u(x_i+\rho y)\det(\nabla\tau(x_i+\rho y))- u(x_i+\rho y)\Big)\,dy\\
&=\sum_{i\in\{0,1\}}\int_{B_1(0)\cap\left(\frac{\Omega-x_i}{\rho}\right)} {(-1)^i\,\frac{u(x_i+\rho y)}{\rho}\rho}\,\kappa\,\varphi'(|y|)\frac{y}{|y|}\cdot\nu_{x_i}\,dy{+o(\kappa)}={o_\rho(1)+o(\kappa)},
\end{aligned}
\]
{where we performed the change of variable $x=x_i+\rho y$, we exploited~\eqref{eq:espdet} and used Theorem~\ref{thm:agalcathm2}, points $(iii)$ and $(iv)$.}

\noindent
Next we deal with the Riesz energy term $ V_\alpha$. We are going to show that
\begin{equation}\label{variazioneriesz}
\frac{1}{\rho^N}\left( V_\alpha(\Omega_\rho)- V_\alpha(\Omega)\right)= \kappa(v(x_1)-v(x_0)) C(\varphi)+o(k)+o_\rho(1),
\end{equation}
where $C(\varphi)$ is the constant defined in \eqref{rfk}.
The proof  of this variation is longer than the previous ones. Let us denote by $v_\rho(\cdot)=\chi_{\Omega_\rho}*|\cdot|^{\alpha-N}$ the Riesz potential of $\Omega_\rho$, and by $v(\cdot)=\chi_{\Omega}*|\cdot|^{\alpha-N}$ the Riesz potential of $\Omega$.
We have
\begin{equation}\label{strokes}
\begin{aligned}
\frac{1}{\rho^N}\left( V_\alpha(\Omega_\rho)- V_\alpha(\Omega)\right)&=\frac{1}{\rho^N} \left(\int_{\Omega_\rho}\vr(x)\,dx-\int_{\Omega}v(x)\,dx\right) \\
&=\frac{1}{\rho^N}\int_\Omega \Big(\vr(\tau(x))\det(\nabla\tau(x))-v(x)\Big)\,dx\\
&=\frac{1}{\rho^N}\int_{\Omega\setminus(B_\rho(x_0)\cup B_\rho(x_1))}(\vr(x)-v(x))\,dx \\
&+\frac{1}{\rho^N} \sum_{i=0,1}\int_{\Omega\cap B_\rho(x_i)}\Big(\vr(\tau(x))\det(\nabla\tau(x))-v(x)\Big)\,dx.
\end{aligned}
\end{equation}
We compute the last two addends of the previous formula separately:
\[
\begin{aligned}
\frac{1}{\rho^N} &\int_{\Omega\cap B_\rho(x_0)}\Big(\vr(\tau(x))\det(\nabla\tau(x))-v(x)\Big)\,dx\\
&=\int_{B_1(0)\cap\left(\frac{\Omega-x_0}{\rho}\right)}\Big(\vr(\tau(x_0+\rho y))\det({\nabla\tau(x_0+\rho y))}-v(x_0+\rho y)\Big)\,dy \\
&=\int_{B_1(0)\cap\left(\frac{\Omega-x_0}{\rho}\right)}\Big(\vr(\tau(x_0+\rho y))(1+\kappa\frac{y}{|y|}\cdot\nu_{x_0}\varphi'(|y|)   +o(\kappa) )-v(x_0+\rho y)   \Big)\,dy\\
&=\int_{B_1(0)\cap\left(\frac{\Omega-x_0}{\rho}\right)}\Big(\vr(\tau(x_0+\rho y))-v(x_0+\rho y)\Big)\,dy \\
&+ \int_{B_1(0)\cap\left(\frac{\Omega-x_0}{\rho}\right)}\Big(\vr(\tau(x_0+\rho y))\,\kappa\frac{y}{|y|}\cdot\nu_{x_0}\varphi'(|y|)\Big)\,dy.
\end{aligned}
\]
First of all we focus on the first term of the chain of inequalities above.
By Lemma \ref{le:rieszregularity}, and by Ascoli-Arzel\`a Theorem,  $v_\rho$ uniformly converges in $B_1(0)$ to some function $\widetilde v$ as $\rho\to0$, and, since its pointwise limit is $v$, we have that $\widetilde v=v$. 
As a consequence, using also Lemma~\ref{le:rieszregularity} and the dominate convergence Theorem, we have \[
\begin{split}
&\int_{B_1(0)\cap\left(\frac{\Omega-x_0}{\rho}\right)}|\vr(\tau(x_0+\rho y))-v(x_0+\rho y)|\,dy\\
&\leq \int_{B_1(0)\cap\left(\frac{\Omega-x_0}{\rho}\right)}|\vr(\tau(x_0+\rho y))-\vr(x_0+\rho y)|+|\vr(x_0+\rho y)-v(x_0+\rho y)|\,dy\\
&\leq \int_{B_1(0)\cap\left(\frac{\Omega-x_0}{\rho}\right)}|k\rho\varphi(|y|)\nu_{x_0}|^{1+\gamma}\,dy+o_\rho(1)\to 0,
\end{split}
\]
as $\rho\to 0$.
Moreover, since  $\chi_{B_1(0)\cap\left(\frac{\Omega-x_0}{\rho}\right)}\to \chi_{B_1(0)\cap\{x\cdot\nu_{x_0}>0\}}$ (see Theorem~\ref{thm:agalcathm2} (iii)), we have that
\[
\begin{split}
&\lim_{\rho\to 0}\int_{B_1(0)\cap\left(\frac{\Omega-x_0}{\rho}\right)}\Big(\vr(\tau(x_0+\rho y))\,\kappa\frac{y}{|y|}\cdot\nu_{x_0}\varphi'(|y|)\Big)\,dy\\
& =v(x_0)\kappa\int_{B_1(0)\cap\{x\cdot\nu_{x_0}>0\}}\nu_{x_0}\cdot\frac{y}{|y|}\varphi'(|y|)\,dy={-\kappa C(\varphi)v(x_0)},
\end{split}
\]
as $\rho\to0$, where {we have used~\eqref{rfk}}, the fact that \[
|\vr(\tau(x_0+\rho y))-v(x_0)|\leq |\vr(\tau(x_0+\rho y))-\vr(x_0)|+|\vr(x_0)-v(x_0)|\to 0,
\]
uniformly on the compact sets and, again,  the dominate convergence Theorem.
A completely analogous  computation shows  that
\[
\frac{1}{\rho^N} \int_{\Omega\cap B_\rho(x_1)}\Big(\vr(\tau(x))\det(\nabla\tau(x))-v(x)\Big)\,dx  \to {\kappa C(\varphi)v(x_1)},
\] 
as $\rho\to0$.

We wish to show now that the first addend on the right-hand side of \eqref{strokes} converges to $0$ as $\rho\to0$. To this aim, we compute
\begin{equation}\label{cage}
\begin{aligned}
\frac{\vr(x)-v(x)}{\rho^N}&=\frac{1}{\rho^N}\left(\int_{\Omega_\rho}\frac{dy}{|x-y|^{N-\alpha}}-\int_{\Omega}\frac{dy}{|x-y|^{N-\alpha}}\right)\\
&=\frac{1}{\rho^N}\int_\Omega\left(\frac{\det(\nabla\tau(y))}{|x-\tau(y)|^{N-\alpha}}  - \frac{1}{|x-y|^{N-\alpha}}  \right)\,dy\\
&=\sum_{i\in\{0,1\}} \frac{1}{\rho^N}\int_{\Omega\cap B_\rho(x_i)}\left(\frac{\det(\nabla\tau(y))}{|x-\tau(y)|^{N-\alpha}}  - \frac{1}{|x-y|^{N-\alpha}}  \right)\,dy\\
&=\sum_{i\in\{0,1\}} \int_{B_1(0)\cap\left(\frac{\Omega-x_i}{\rho}\right)}\left(\frac{1}{|x-\tau(x_i+\rho y)|^{N-\alpha}}  - \frac{1}{|x-(x_i+\rho y)|^{N-\alpha}}  \right)\,dy \\
&+\sum_{i\in\{0,1\}}(-1)^i\int_{B_1(0)\cap\left(\frac{\Omega-x_i}{\rho}\right)}\frac{\kappa\varphi'(|y|)\nu_{x_i}\cdot\frac{y}{|y|}}{|x-\tau(x_i+\rho y)|^{N-\alpha}}\,dy +o(\kappa).
\end{aligned}
\end{equation}
We remark that the last two addends converge to the same constant, with opposite sign. Thus in the limit they elide themselves:
\[
\sum_{i\in\{0,1\}}(-1)^i\int_{B_1(0)\cap\left(\frac{\Omega-x_i}{\rho}\right)}\frac{\kappa\varphi'(|y|)\cdot\nu_{x_i}\frac{y}{|y|}}{|x-\tau(x_i+\rho y)|^{N-\alpha}}\,dy\to0 \quad\text{as $\rho\to0$}.
\]
Now we notice that for any $X,Y,Z\in\R^N$ it holds that
\[
\frac{1}{|X-Y|^{N-\alpha}}-\frac{1}{|X-Z|^{N-\alpha}}\le (N-\alpha+1)\frac{\min(1,|Y-Z|)}{\min(|X-Y|^{N-\alpha},|X-Y|^{N-\alpha+1})}.
\]
Such an inequality can be proved easily by convexity, see for instance \cite[formula (2.11)]{fuprariesz}.
By applying such an inequality in the first two addends of the right-hand side of \eqref{cage} with $X=x$, $Y=x_i+\rho y$ and $Z=\tau(x_i+\rho y)$ we get that
\[
\begin{aligned}
&\sum_{i\in\{0,1\}} \int_{B_1(0)\cap\left(\frac{\Omega-x_i}{\rho}\right)}\left(\frac{1}{|x-\tau(x_i+\rho y)|^{N-\alpha}}  - \frac{1}{|x-(x_i+\rho y)|^{N-\alpha}}  \right)\,dy\\
&\le C(N,\alpha)\sum_{i\in\{0,1\}} \int_{B_1(0)\cap\left(\frac{\Omega-x_i}{\rho}\right)}\frac{\min(1,|\tau(x_i+\rho y)-(x_i+\rho y)|)}{\min(|x-(x_i+\rho y)|^{N-\alpha},|x-(x_i+\rho y)|^{N-\alpha+1})}\,dy\\
&\le C(N,\alpha)\|\varphi\|_{C^0}\rho\sum_{i\in\{0,1\}} \int_{B_1(0)\cap\left(\frac{\Omega-x_i}{\rho}\right)}\frac{1}{\min(|x-(x_i+\rho y)|^{N-\alpha},|x-(x_i+\rho y)|^{N-\alpha+1})}\,dy\\
& \le C(N,\alpha,\varphi)\rho\sum_{i=0,1} \int_{B_1(0)}\frac{1}{ |x-(x_i+\rho y)|^{N-\alpha}}\,dy+ \int_{B_1(0)}\frac{1}{ |x-(x_i+\rho y)|^{N-\alpha+1}}\,dy.
\end{aligned}
\] 
In the second inequality we used the fact that
\[
\min(1,|\tau(x_i+\rho y)-(x_i+\rho y)|)\le \|\varphi\|_{C^0}\rho.
\]
Since the last two integrals are finite, being $\alpha>1$, we get the desired claim, that is \eqref{variazioneriesz}. 

The conclusion now  readily follows: by minimality of $\Omega$ and thanks to \eqref{variazionevolume}, \eqref{variazionetorsione} and \eqref{variazioneriesz} we have that
\[
\begin{split}
&0\le \G(\Omega_\rho)-\G(\Omega)\\
&{\le \kappa\rho^NC(\varphi)\Big((q_u(x_0)^2-q_u(x_1)^2) + \eps(v(x_1)-v(x_0)) \Big)+o(\rho^N)+\rho^N o(\kappa).}
\end{split}
\]
Since { from the assumptions we have $(q_u(x_0)^2-q_u(x_1)^2) + \eps(v(x_1)-v(x_0))<0$}, by choosing {$\rho$ and $\kappa$ small enough}, we get the desired contradiction. The proof is concluded. 
\end{proof}
An immediate consequence of Lemma~\ref{le:rieszregularity} and Theorem~\ref{thm:optcondqu} is the following.
\begin{corollary}\label{cor:quc2gamma}
Let $\Om$, $u$ and $q_u$ be as above. For some constant $\Lambda_\eps>0$, we have
\begin{equation}\label{eq:optcondbdry}
q_u^2(x)-\eps v_\Om(x)=\Lambda_\eps,\qquad \text{for }x\in \partial^*\Om.
\end{equation}
Moreover, $q_u\in C^{1,\gamma}$ for some $\gamma\in(0,1)$ and \[
\|q_u\|_{C^{1,\gamma}}\leq C(N,\alpha,R).
\]
\end{corollary}
Finally, to prove that the boundary of $\Om_\eps$ is locally the graph of a $C^{2,\gamma}$ function on the boundary of a ball, we only need to implement the improvement of flatness technique from~\cite[Section~7 and~8]{alca}, which can be readapted with minimal changes to our setting as shown in~\cite[Appendix]{gush}.
\begin{definition}\label{def:classF}
Let $\mu_\pm\in(0,1]$ and $k>0$. A weak solution $u$ of~\eqref{eq:bdp4.31} is of class $F(\mu_-,\mu_+,k)$ in $B_\rho(x_0)$ with respect to direction $\nu\in \mathbb{S}^{N-1}$ if 
\begin{enumerate}
\item[(a)]$x_0\in \partial \{u>0\}$ and \[
\begin{split}
u=0,\qquad &\text{for }(x-x_0)\cdot \nu\leq -\mu_-\rho,\quad {x\in B_\rho(x_0)},\\
u(x)\geq q_{u}(x_0)[(x-x_0)\cdot \nu-\mu_+\rho],\qquad &\text{for }(x-x_0)\cdot \nu\geq \mu_+\rho,\quad {x\in B_\rho(x_0)}.
\end{split}
\]
\item[(b)] $|\nabla u(x_0)|\leq q_u(x_0)(1+k)$ in $B_\rho(x_0)$ and ${\rm osc}_{B_{\rho}(x_0)}q_u\leq kq_u(x_0)$.
\end{enumerate} 
\end{definition}
We note that if $k=+\infty$, then condition $(b)$ is automatically satisfied, that is, no bounds on the gradient are required.
The fact that our minimizers are nearly spherical sets of class $C^{2,\gamma}$ is now a direct consequence of the  following regularity result, which was first proved in~\cite[Theorem~8.1]{alca} and \cite[Theorem~2]{kini}.
\begin{theorem}\label{thm:bdvthm4.18}
Let $u$ be a weak solution to~\eqref{eq:bdp4.31} in $B_R$ and assume that $q_{u}$ is $C^{1,\gamma}$ for some constant $\gamma\in(0,1)$ in a neighborhood of $\{u>0\}$. Then there are constants $\overline \mu$ and $\overline k$, depending only on $N$, $\alpha$, $R$, $\max q_u$, $\min q_u$, $\|q_u\|_{C^{1,\gamma}}$ such that: 

If $u$ is of class $F(\mu,1,+\infty)$ in $B_{4\rho}(x_0)$ with respect to some direction $\nu\in \mathbb{S}^{N-1}$ with $\mu\leq \overline \mu$ and $\rho\leq \overline k \mu^2$, then there exists a $C^{2,
\gamma}$ function $f\colon \R^{N-1}\to \R$ with $\|f\|_{C^{2,\gamma}}\leq C(N,\alpha,R, \|q_u\|_{C^{1,\gamma}})$ such that, calling \[
{\rm graph}_\nu f:=\{x\in \R^N : x\cdot \nu =f(x-(x\cdot\nu)\nu)\},
\]
then \[
\partial \{u>0\}\cap B_\rho(x_0)=(x_0+{\rm graph}_\nu(f))\cap B_\rho(x_0).
\]
\end{theorem}

%\newpage

%%%%%%%%%%%%%%%%%%%%%
%CONCLUSION%
%%%%%%%%%%%%%%%%%%%%%

%\newpage
\section{Proof of Theorem \ref{thm:mainvero2}}\label{mainthm}
In the last section we have shown that any minimizer for problem \eqref{eq:mintor} has boundary close to that of a ball (precisely, the ball which achieve the minimum in the definition of asymmetry), and is locally $C^{2,\gamma}-$ regular. This, reasoning as in ~\cite[Proof of Proposition~4.4]{brdeve}, is enough to show that such a minimum is a nearly spherical set, and to conclude the proof of Theorem  \ref{thm:mainvero2}.
\begin{proof}[Proof of Theorem~\ref{thm:mainvero2}]
Thanks to Theorem~\ref{thm:noconstraint} and Lemma~\ref{le:existminG}, for $\eps_{*}$ small enough (depending on $N,\alpha,R$), there is a minimizer $\Om_\eps$ for~\eqref{eq:mintor} and we can assume without loss of generality that the barycenter of $\Om_\eps$ is $x_{\Om_\eps}=0$. 
It is not difficult to show  that the sequence of the translated sets $\Om_\eps$ with barycenter at the origin still converges in $L^1$ to the ball $B$ of unit measure and centered at the origin, and thus the statement of Lemma~\ref{le:lemma5.1} applies for them.
We call $u_\eps$ the torsion function of $\Om_\eps$, so that $\Omega_\eps=\{u_\eps>0 \}$.
We claim that $\Om_\eps$ is a $C^{2,\gamma}$ nearly spherical set.
To see this, let $\overline k,\overline \mu$ be as in Theorem~\ref{thm:bdvthm4.18} and $\mu<\overline \mu$ to be fixed later. Since $\partial B$ is smooth, there exists $\rho(\mu)\leq \overline k \mu^2$ such that, for all $\rho\leq \rho(\mu)$ and all $\overline x\in \partial B$, we have \[
\partial B\cap B_{5\rho}(\overline x)\subset \Big\{x : |(x-\overline x)\cdot \nu_{\overline x}|\leq \mu\rho\Big\},
\]
where hereafter $\nu_{\overline x}$ is the { inner} unit normal to $\partial B$ at $\overline x$.
By Lemma~\ref{le:lemma5.1}, up to take $\eps_E$ small enough (depending possibly also on $\mu$), there is a point $x_0\in \partial \Om_\eps\cap B_{\mu\rho(\mu)}(\overline x)$ such that \[
\partial \Om_\eps\cap B_{4\rho(\mu)}(x_0)\subset B_{\mu\rho(\mu)}\Big(\partial B\cap B_{5\rho(\mu)}(\overline x)\Big)\subset \Big\{x : |(x-x_0)\cdot \nu_{\overline x}|\leq 4\mu\rho(\mu)\Big\}.
\]
{ We notice that, with the notation of Definition~\ref{def:classF}, the second condition of part $(a)$ holds if $\mu_+=1$, since $u_\eps\geq0$. Therefore }$u_\eps$ is of class $F(\mu,1,+\infty)$ in ${B_{4\rho(\mu)}(x_0)}$ in direction $\nu_{\overline x}$  and hence, by Theorem~\ref{thm:bdvthm4.18} and Corollary~\ref{cor:quc2gamma}, we infer that ${\partial \Om_\eps\cap B_{\rho(\mu)}(x_0)}$ is the graph of a $C^{2,\gamma}$ function with respect to $\nu_{\overline x}$.
So, up to further decrease $\mu$, there are functions $\varphi_\eps^{\overline x}$ with $C^{2,\gamma}$ norm uniformly bounded such that \[
\partial \Om_\eps\cap B_{\rho(\mu)}(\overline x)=\Big\{x+\varphi_\eps^{\overline x}(x)x : x\in \partial B\Big\}.
\]
As the balls $\{B_{\rho(\mu)}(\overline x)\}_{\overline x\in \partial B}$ cover $\partial B$, by compactness there is a function $\varphi_\eps\in C^{2,\gamma}(\partial B)$ with bounded $C^{2,\gamma}$ norm.
Moreover, up to take $\eps_E$ small enough, by Lemma~\ref{le:lemma5.1}, we can assume that $\|\varphi_\eps\|_{C^{2,\gamma'}}$ is as small as we wish. A direct application of~\cite[Theorem~3.3]{brdeve} (recalling also that $\Om_\eps$ has barycenter in the origin) entails  that \[
E(\Om_\eps)-E(B)\geq C(N)\|\varphi_\eps\|^2_{H^{1/2}(\partial B)}\geq C(N)\|\varphi_\eps\|^2_{L^2(\partial B)}.
\]
Up to further decrease $\eps_E$, by~\cite[equation~(6.8)]{km} we have \[
 V_\alpha(B)- V_\alpha(\Om_\eps)\leq C'(N,\alpha)\|\varphi_\eps\|^2_{L^2(\partial B)}.
\]
By minimality of $\Om_\eps$ and the two bounds above, we have \[
C(N)\|\varphi_\eps\|^2_{L^2(\partial B)}
\leq E(\Om_\eps)-E(B)\leq \eps( V_\alpha(B)- V_\alpha(\Om_\eps))\leq C'(N,\alpha)\eps \|\varphi_\eps\|^2_{L^2(\partial B)}.
\]
Since the constants $C$ and $C'$ are independent of $\eps$, we can take $\eps_E$ small enough (depending on $N,\alpha,R$) so that, for all $\eps\leq \eps_E$ we have \[
E(\Om_\eps)=E(B),
\]
and by the rigidity of the Saint-Venant inequality, we conclude. 
\end{proof}

%\newpage

\section{A surgery result for the functional involving the first eigenvalue}\label{sect:surgery}
In this section we prove the following \emph{surgery} result. Throughout this section, $\Omega$ is an open set of unit measure, $B$ is the ball of unit measure centered at the origin and we define
\[
\widetilde{\mathcal F}_{\alpha,\eps}(\Omega)=\lambda_1(\Omega)+\eps V_\alpha(\Om).
\] 

\begin{proposition}\label{prop:surgery}
Let $\alpha\in(1,N)$.
There exist constants $D(N,\alpha)$, $\overline \delta (N,\alpha)<1$ and $\overline \eps(N,\alpha)$ such that if $\eps\leq \overline \eps(N,\alpha)$ then for any open and connected set $\Om\subset \R^N$ of unit measure satisfying $\la_1(\Om)-\la_1(B)\leq \overline \delta(N,\alpha)$ there exists an open, connected set $\widehat \Om$ of unit measure with diameter bounded by $D$ and such that
 \[
{\widetilde {\mathcal F}}_{\alpha,\eps}(\widehat \Om)\leq {\widetilde {\mathcal F}}_{\alpha,\eps}(\Om).
\]
\end{proposition}
The proof of the proposition is quite technical and is  mostly inspired by~\cite{mp} (see also~\cite{buma}). We have skipped the proofs that are essentially identical, while we have detailed the points where substantial changes need to be made.
\begin{remark}\label{rmk:surgeryconnected}
{\bf On the analogies and differences with respect to \cite{mp}.}
The connectedness assumption is a main difference with respect to the work in~\cite{mp}, though it does not change much the argument.
The reason for which we need to impose it is the presence in our functional of the repulsive Riesz potential energy. 
On the other hand this difficulty is compensated by the fact that, by choosing $\eps$ small, we can arbitrarily impose that the sets we take into account have small Fraenkel asymmetry. Moreover, dealing with only the first eigenvalue simplifies many technical steps related to the orthogonality of the higher eigenfunctions.
\end{remark}
Let us  introduce some notation.
Let $\Omega$ be a connected set such that  with $\la_1(\Om)-\la_1(B)\leq \overline \delta(N,\alpha)$, so that, by the quantitative Faber-Krahn inequality (see Theorem~\ref{thm:quantitativefk}),  up to translations we have \[
|\Om\Delta B|=\mathcal A(\Om)\leq \left(\frac{\overline \delta}{\widehat \sigma}\right)^{1/2},
\]
where $\widehat \sigma=\widehat \sigma(N)$.
From now on we fix $\Om$ so that $B$ is the ball of unit measure attaining the asymmetry and we will no more translate it.
By defining $K=K(N):=\la_1(B)+1\geq \la_1(B)+\overline \delta(N,\alpha)$ we get immediately 
 \[
\la_1(\Om)\leq K.
\] 
We then call $\overline t :=\left(\frac{1}{\omega_N}\right)^{1/N}$ the radius of the ball $B$ and note that\[
|\Om\setminus [- t,t]^N|\leq |\Om\Delta B|= \mathcal A(\Om),\qquad \text{for all }t\geq \overline t.
\]
%%
%Moreover we define a positive number $\mh=\mh(N,\alpha):=|\Om\setminus \{(t,y)\in \R^N : t\geq -\bar t\}$.
%
Let $\mh\in (0, 1/4)$ be such that
\begin{equation}\label{defmh}
\frac{(4\mh)^{\frac 2N}}{\lambda_1(B)} \, K \leq \frac 12\,.
\end{equation}
Moreover, we choose $\overline \delta(N,\alpha)$ small enough so that 
\begin{equation}\label{eq:asimmpiccolamhat}
|\Om\setminus [-\bar t ,\bar t ]^N|\leq \mathcal A(\Om)\leq \sqrt{\frac{\overline \delta}{\widehat \sigma}}\leq\frac{\mh}{2^{2N}}.
\end{equation}
We first focus on the direction $e_1$ and detail the construction in this case.
We shall denote $z=(x,y)\in \R\times\R^{N-1}$ and by $z_i$ the $i$-th component of $z\in \R^N$.
For any $t\in \R$, we  define
\begin{equation*}
\Om_t:=\Big\{y\in\rbb^{N-1} : (t,y)\in\Om\Big\}\,,
\end{equation*}
and given any set $\Omega\subseteq\R^N$, we define its $1$-dimensional projections for $1\leq p \leq N$ as
\[
\pi_p(\Omega) := \Big\{ t\in\R:\, \exists \, (z_1,\, z_2,\, \dots\, ,\, z_N)\in\Omega,\, z_p = t\Big\}\,.
\]
For every $t\leq -\bar t$ we call
\begin{align}\label{int0}
\Omega^+(t) := \Big\{(x,y)\in\Om : x>t\Big\}\,, && \Omega^-(t) := \Big\{(x,y)\in\Om : x<t\Big\}\,, && \eps(t):=\hc^{N-1}(\Om_t)\,.
\end{align}
Observe that
\begin{equation}\label{int1}
m(t) := \big| \Omega^-(t) \big| = \int_{-\infty}^t \eps(s)\,ds\leq 2\mh\,.
\end{equation}
We call $u$ the first eigenfunction on $\Om$ with unit $L^2$ norm. 
We define then also, for every $t\leq \bar t$,
\begin{align}\label{int2}
\de(t):=\int_{\Om_t}{|\nabla u(t,y)|^2\,d\hc^{N-1}(y)}\,, && \mu(t):=\int_{\Om_t}{u(t,y)^2\,d\hc^{N-1}(y)}\,,
\end{align}
which makes sense since $u$ is smooth inside $\Om$. It is convenient to give the further notation
\begin{equation}\label{defphi}
\phi(t) :=  \int_{\Omega^-(t)}  |\nabla u|^2 = \int_{-\infty}^t \delta(s)\,ds \,.
\end{equation}
Applying the Faber--Krahn inequality in $\R^{N-1}$ to the set $\Omega_t$, and using the rescaling property of eigenvalues on $\R^{N-1}$, we know that
\[
\eps(t)^{\frac{2}{N-1}} \lambda_1(\Omega_t)=\H^{N-1}(\Omega_t)^{\frac{2}{N-1}} \lambda_1(\Omega_t) \geq \lambda_1(B_{N-1})\,,
\]
calling $B_{N-1}$ the unit ball in $\R^{N-1}$. As a trivial consequence, we can estimate $\mu$ in terms of $\eps$ and $\de$: in fact, noticing that $u(t,\cdot)\in H^{1}_0(\Omega_t)$ and writing $\nabla u = (\nabla_1 u, \nabla_y u)$, we have
\begin{equation}\label{muest}
\mu(t)=\int_{\Om_t}{u(t,\cdot)^2\,d\hc^{N-1}}\leq \frac{1}{\la_1(\Om_t)}\int_{\Om_t}{|\nabla_y u(t,\cdot)|^2\,d\hc^{N-1}}\leq C\eps(t)^{\frac{2}{N-1}}\de(t).
\end{equation}
We can now present two estimates which assure that $u$ and $\nabla u$ can not be too big in $\Om^-(t)$.

\begin{lemma}\label{primastima}
Let $\Omega\subseteq \R^N$ be an open and connected set of unit volume and with $\lambda_1(\Omega)\leq K$. For every $t\leq -\bar t$ the following inequalities hold: 
\begin{align}\label{udu-}
\int_{\Om^-(t)} u^2 \leq C_1 \eps(t)^{\frac{1}{N-1}}\de(t)\,, &&
\int_{\Om^-(t)} |\nabla u|^2 \leq C_1 \eps(t)^{\frac{1}{N-1}}\de(t)\,,
\end{align}
for some $C_1=C_1(N)$ (recalling that $K$ for us is a precise constant depending only on $N$).
\end{lemma}
The proof of the above Lemma follows exactly as in~\cite[Lemma~2.3]{mp}.

Let us go further into the construction, giving some additional definitions. For any $t\leq -\bar t$ and $\sigma(t)>0$, we define the cylinder $Q(t)$ as
\begin{equation}\label{defcyl}
Q(t):=\Big\{(x,y)\in \rbb^N: \, t-\sigma(t) < x < t,\ (t,y) \in\Om\Big\} = \big(t-\sigma(t),t\big) \times \Omega_t\,,
\end{equation}
where for any $t\leq -\bar t$ we set
\begin{equation}\label{defsigma}
\sigma(t)= \eps(t)^{\frac 1{N-1}}\,.
\end{equation}
We let also $\Omt(t)=\Omega^+(t)\cup Q(t)$, and we introduce $\ut\in H^1_0\big(\Omt(t)\big)$ as
\begin{equation}\label{utilde}
\ut(x,y):=\left\{
\begin{array}{ll}
u(x,y) &\hbox{if $(x,y)\in \Omega^+(t)$}\,, \\[5pt]
\bal\frac{x-t+\sigma(t)}{\sigma(t)}\,u(t,y)\eal &\hbox{if $(x,y)\in Q(t)$}\,.
\end{array}
\right.
\end{equation}
The fact that $\ut$ vanishes on $\partial\Omt(t)$ is obvious; moreover, $\nabla u=\nabla \ut$ on $\Omega^+(t)$, while on $Q(t)$ one has
\begin{equation}\label{estdut}
\nabla \ut(x,y)=\left(\frac{u(t,y)}{\sigma(t)}\,,\,\frac{x-t+\sigma(t)}{\sigma(t)}\,\nabla_y u(t,y)\right)\,.
\end{equation}

A simple calculation allows us to estimate the integrals of $\ut$ and $\nabla \ut$ on $Q(t)$.
\begin{lemma}\label{lemmatest}
For every $t\leq -\bar t$, one has
\begin{align}\label{newtest}
\int_{Q(t)}|\nabla \ut|^2 \leq C_2\eps(t)^{\frac{1}{N-1}}\de(t) \,, &&
\int_{Q(t)} \ut^2 \leq C_2 \eps(t)^{\frac 3{N-1}} \delta(t)\,,
\end{align}
for a suitable constant $C_2=C_2(N)$.
\end{lemma}
The proof of the above Lemma follows as~\cite[Lemma~2.4]{mp}.

Another simple but useful estimate concerns the Rayleigh quotients of the functions $\ut$ on the sets $\Omt(t)$.

\begin{lemma}\label{noth}
There exists a constant $C_3=C_3(N)$ such that for every $t\leq -\bar t$, one has
\begin{equation}\label{est1a}
\la_1(\Omt(t))\leq \rc\big(\ut,\Omt(t)\big)\leq \la_1(\Om)+ C_3\eps(t)^{\frac 1 {N-1}}\de(t)\,.
\end{equation}
\end{lemma}
The proof of the above Lemma follows as in~\cite[Lemma~2.5]{mp}, but it is actually simpler since in our setting only the first eigenfunction is involved and we do not need to take care of orthogonality constraints.

We can now enter in the central part of our construction. Basically, we aim to show that either $\Omega$ already has bounded left ``tail'' in direction $e_1$, or some rescaling of $\Omt(t)$ has energy lower than that of $\Omega$. 
\begin{lemma}\label{threeconditions}
Let $\Omega$ be as in the assumptions of Lemma~\ref{primastima}, and let $t\leq -\bar t$. There exist $\overline \eps=\overline \eps(N,\alpha)$ and $C_4=C_4(N,\alpha)>2$ such that, for all $\eps\leq \overline \eps$ exactly one of the three following conditions hold:
\begin{enumerate}
\item[(1) ] $\max\big\{ \eps(t),\, \delta(t) \big\} > 1$;
\item[(2) ] (1) does not hold and $m(t) \leq C_4 \big( \eps(t) + \delta(t)\big) \eps(t)^{\frac 1{N-1}}$;
\item[(3) ] (1) and~(2) do not hold and one has that $\la_1(\Omh(t))\leq \la_1(\Om)$  and\[{\widetilde {\mathcal F}}_{\alpha,\eps}\big(\Omh(t)\big)< {\widetilde {\mathcal F}}_{\alpha,\eps}(\Omega),
\] where for $t\leq -\bar t$  we set
$
\Omh(t) := \big| \Omt(t) \big|^{-\frac 1N} \Omt(t).
$
\end{enumerate}
\end{lemma}
\begin{proof}
Assume~(1) is false. Then it is possible to apply Lemma~\ref{noth}, to get
\begin{equation}\label{putinto}
\la_1\big(\Omt(t)\big)\leq\la_1(\Om)+C_3\eps(t)^{\frac 1 {N-1}}\delta(t)\,.
\end{equation}
By the scaling properties of the eigenvalue and the fact that $\big| \Omh(t)\big|=1$, we know that
\[
\lambda_1\big( \Omh(t)\big)= \big| \Omt(t) \big|^{\frac 2N} \lambda_1\big(\Omt(t)\big)\,.
\]
By construction,
\[
\big| \Omt(t)\big|= \big| \Om^+(t)\big|+ \big|Q(t)\big| = 1 - m(t) + \eps(t)^{\frac N{N-1}}\,,
\]
hence the above estimates and~\eqref{putinto} lead to
\begin{equation}\label{muchhere}\begin{split}
\lambda_1(\Omh(t)) &= \Big( 1 - m(t) + \eps(t)^{\frac N{N-1}} \Big)^{\frac 2N} \lambda_1\big(\Omt(t)\big)\\
&\leq \Big( 1 - \frac 2N \, m(t) + \frac 2N\, \eps(t)^{\frac N{N-1}} \Big)  \Big( \la_1(\Om)+C_3\eps(t)^{\frac 1 {N-1}}\delta(t) \Big)\\
&\leq \lambda_1(\Omega) - \frac{2\lambda_1(B)}{N}\, m(t)+\frac{2K}N\, \eps(t)^{\frac N{N-1}}+\bigg(2C_3+\frac 2N\bigg)\eps(t)^{\frac 1 {N-1}}\delta(t)\,.
\end{split}
\end{equation}
At this point, defining $C_4:= \max{\{\frac{2(K+1)}N+2C_3,2\}}$, if \[
m(t) \leq C_4 \big( \eps(t) + \delta(t)\big) \eps(t)^{\frac 1{N-1}},
\] 
then condition~(2) holds true. 
Otherwise, we immediately have that 
\begin{equation}\label{eq:stimala1}
\la_1(\Omh(t))\leq \la_1(\Om)-\left(\frac{2\la_1(B)}{N}-1\right)m(t)=\la_1(\Om)-C_5(N)m(t),
\end{equation}
for a constant $C_5>0$, therefore the first part of the third claim is verified.

Moreover, we can compute, using Lemma~\ref{le:knupfermuratov}, the Riesz inequality and noting that $|\Om\Delta \Omt(t)|\leq m(t)+\eps(t)^{\frac{N}{N-1}}$,  
\begin{equation}\label{eq:stimaV}
\begin{split}
&V_\alpha(\Omh(t))\leq V_\alpha (\Omt(t))\left(1-m(t)+\eps(t)^{\frac{N}{N-1}}\right)^{-\frac{N+\alpha}{N}}\leq \Big(1+2\frac{N+\alpha}{N}m(t)\Big)V_\alpha(\Omt(t))\\
&\leq  \Big(1+2\frac{N+\alpha}{N}m(t)\Big)\Big(V_\alpha(\Om)+C_0(N)|\Om\Delta \Omt(t)|\left[|\Om|^{\frac{\alpha}{N}}+|\Omt(t)|^{\frac{\alpha}{N}}\right]\Big)\\
&\leq V_\alpha(\Om) +2\frac{N+\alpha}{N}V_\alpha(B)m(t)+2C_0(m(t)+\eps(t)^{\frac{N}{N-1}})+8C_0\frac{N+\alpha}{N}m(t)\\
&\leq V_\alpha(\Om) +2\frac{N+\alpha}{N}V_\alpha(B)m(t)+2C_0(1+\frac{1}{C_4})m(t)+8C_0\frac{N+\alpha}{N}m(t)\\
&=V_\alpha(\Om)+C_6(N,\alpha) m(t).
\end{split}
\end{equation}
Then, putting together~\eqref{eq:stimala1} and~\eqref{eq:stimaV}, we deduce 
\begin{equation}
\begin{split}
&\lambda_1(\Omh(t))+\eps V_\alpha(\Omh(t))\\
&\leq \la_1(\Om)+\eps V_\alpha(\Om)-(C_5-\eps C_6)m(t)\\
&\leq \la_1(\Om)+\eps V_\alpha(\Om)-\frac{C_5}{2}m(t),
\end{split}
\end{equation}
up to take $\eps\leq \overline \eps(N,\alpha)<\frac{C_5}{2C_6}$,  so that in this case condition $(3)$ holds
and the proof is concluded.
\end{proof}

\begin{lemma}\label{lemmatail}
Let $\alpha\in(1,N)$.
For every $\eps\leq \overline \eps(N,\alpha)$, and for any open and connected set $\Omega\subseteq \R^N$ of unit volume, with $\lambda_1(\Omega)\leq K$ and \[
|\Om\cap \{(x,y)\in\R\times\R^{N-1} : x<-\bar t\}|\leq \mh,
\] 
there exists another open, connected set $U_1^-\subseteq\R^N$, still of unit volume, such that 
\begin{enumerate}
\item $\la_1(U_1^-)\leq \la_1(\Om)$ and ${\widetilde {\mathcal F}}_{\alpha,\eps}(U_1^-)\leq{\widetilde {\mathcal F}}_{\alpha,\eps}(\Om)$,
\item $U_1^-\subset \Big\{(x,y)\in \R\times\R^{N-1} : x>-2C_7-4-2\bar t\Big\}$,
\item $|U_1^-\setminus [-2\bar t,2\bar t]^N|\leq \frac{\mh}{2^{2N-1}}\leq \mh$.
\end{enumerate}
\end{lemma}

\begin{proof}
Let us start defining
\begin{equation}\label{defhatt}
\hat t := \sup \Big\{ t \leq -\bar t : \, \hbox{condition~(3) of Lemma~\ref{threeconditions} holds for $t$}\Big\}\,,
\end{equation}
with the usual convention that, if condition~(3) is false for every $t\leq -\bar t$, then $\hat t = -\infty$. We introduce now the following subsets of $(\hat t,-\bar t)$,
\begin{align*}
A:&= \Big\{ t\in (\hat t,-\bar t):\, \hbox{condition~(1) of Lemma~\ref{threeconditions} holds for $t$} \Big\}\,,\\
B:&=\Big\{t\in (\hat t,-\bar t):\, \hbox{condition~(2) of Lemma~\ref{threeconditions} holds for $t$} \Big\}\,,
%il fatto che $m(t)>0$ e' ovvio perche' lavoriamo su insiemi connessi.
\end{align*}
and we further subdivide them as
\begin{align*}
A_1:= \Big\{ t \in A:\, \eps(t)\geq \delta(t) \Big\}\,, && A_2:= \Big\{ t \in A:\, \eps(t)<\delta(t) \Big\}\,, \\
B_1:= \Big\{ t \in B:\, \eps(t)\geq \delta(t) \Big\}\,, && B_2:= \Big\{ t \in B:\, \eps(t)<\delta(t) \Big\}\,.
\end{align*}
We aim to show that both $A$ and $B$ are uniformly bounded. Concerning $A_1$, observe that
\[
\big|A_1\big| \leq \int_{A_1} \eps(t)\, dt = \Big| \Big\{ (x,y)\in \Omega:\, x \in A_1\Big\}\Big| \leq \big| \Omega\big|=1\,,
\]
so that $|A_1|\leq 1$. Concerning $A_2$, in the same way and also recalling that $\lambda_1(\Omega)\leq K$, we have
\[
\big|A_2\big| \leq \int_{A_2} \delta(t)\, dt 
=  \int_{A_2 }\int_{\Omega_t} \big|\nabla u(t,y)\big|^2\, d\H^{N-1}(y)\,dt 
\leq  \int_\Omega \big|\nabla u\big|^2 \leq K\,,
\]
so that $|A_2|\leq K$. Summarizing, we have proved that
\begin{equation}\label{estA}
\big| A \big| \leq {1 + K}\,.
\end{equation}
Let us then pass to the set $B_1$. To deal with it, we need a further subdivision, namely, we write $B_1= \cup_{n\in\N} B_1^n$, where
\begin{equation}\label{defB1n}
B_1^n := \bigg\{ t\in B_1:\, \frac{\mh}{2^n} < m(t) \leq \frac{\mh}{2^{n-1}} \bigg\}\,.
\end{equation}
We note that it is possible that some of the $B_1^n$ are empty, in particular this happens for $n< 2N-1$, because $m(t)\leq |\Om\setminus [-\bar t,\bar t]^N|\leq \frac{\mh}{2^{2N}}$, but this does not affect our argument.
Keeping in mind~\eqref{int1}, we know that $t\mapsto m(t)$ is an increasing function, and that for a.e. $t \in\R$ one has $m'(t) = \eps(t)$. Moreover, for every $t\in B_1$ one has by construction that
\[
m(t)\leq C_4 \big( \eps(t) + \delta(t)\big) \eps(t)^{\frac 1{N-1}}\leq 2 C_4\, \eps(t)^{\frac N{N-1}}\,.
\]
As a consequence, for every $t\in B_1^n$ one has
\[
m'(t) = \eps(t) \geq \frac{1}{C}\, m(t)^{\frac{N-1}N}
\geq \frac{1}{C}\, \mh^{\frac{N-1}N} \, \frac{1}{\big(2^{\frac{N-1}N}\big)^n}\,.
\]
This readily implies
\[
\frac{1}{C}\, \mh^{\frac{N-1}N} \, \frac{1}{\big(2^{\frac{N-1}N}\big)^n} \, \big| B_1^n \big| 
\leq \int_{B_1^n} m'(t) \leq \frac{\mh}{2^n}\,,
\]
which in turn gives
\[
\big| B_1^n \big| \leq  C \mh^{\frac 1N} \big(2^{-\frac 1N}\big)^n\,.
\]
Finally, we deduce
\begin{equation}\label{estB1}
\big| B_1 \big| = \sum_{n\in\N} \big| B_1^n\big| 
\leq C \mh^{\frac 1N} \sum_{n\in\N} \big(2^{-\frac 1N}\big)^n
= C \mh^{\frac 1N} \,\frac{2^{\frac 1N}}{2^{\frac 1N}-1}\,.
\end{equation}

Concerning $B_2$, we can almost repeat the same argument: in fact, thanks to~\eqref{udu-}, for every $t\in B_2$ we have
\begin{align*}
\phi(t)=\int_{\Om^-(t)}|\nabla u|^2 \leq C_1\,\eps(t)^{\frac{1}{N-1}}\delta(t)\leq C_1\,\delta(t)^{\frac{N}{N-1}}\,, && 
\hbox{with } \delta(t) = \phi'(t)\,.
\end{align*}
which is the perfect analogous of the above setting with $\delta$ and $\phi$ in place of $\eps$ and $m$ respectively. Since as already observed $\phi(\bar t) \leq  \int_\Omega |\nabla u|^2 \leq K$, in analogy with~\eqref{defB1n} we can define
\[
B_2^n := \bigg\{ t\in B_2:\, \frac{K}{2^{n+1}} < \phi(t) \leq \frac{K}{2^n} \bigg\}\,,
\]
thus the very same argument which leads to~\eqref{estB1} now gives
\begin{equation}\label{estB2}
\big| B_2 \big| = \sum_{n\in\N} \big| B_2^n\big| 
\leq C \big(K\big)^{\frac 1N} \sum_{n\in\N} \big(2^{-\frac 1N}\big)^n
= C \big(K\big)^{\frac 1N} \,\frac{2^{\frac 1N}}{2^{\frac 1N}-1}\,.
\end{equation}
Putting~\eqref{estA}, \eqref{estB1} and~\eqref{estB2} together, we find
\begin{equation}\label{almostR_1}
\big| A \big| +\big| B \big| \leq C_7 = C_7(N,\alpha)\,.
\end{equation}
We need now to distinguish two cases for $\Om$.

\case{I}{One has $\hat t=-\infty$.}
If this case happens, then condition~(3) of Lemma~\ref{threeconditions} never holds true, i.e., for every $t\leq -\bar t$ either condition~(1) or~(2) holds. Recalling the definition of $A$ and $B$ and~\eqref{almostR_1}, we deduce that, choosing simply $U_1^-=\Om$, \[
U_1^-\subset \Big\{(x,y)\in \R\times \R^{N-1} : x>-C_7-\bar t\Big\}\subset \Big\{(x,y)\in \R\times \R^{N-1} : x>-2C_7-4-2\bar t\Big\}
\] 
Therefore, the remaining parts of the claim of Lemma~\ref{lemmatail} is immediately obtained, noting that clearly \[
|U_1^-\setminus [-2\bar t,2\bar t]^N|\leq |\Om\setminus [-\bar t ,\bar t]^N|\leq \frac{\mh}{2^{2N}}\leq \frac{\mh}{2^{2N-1}}.
\]
\par

\case{II}{One has $\hat t>-\infty$.}
In this case, let us notice the connectedness of $\Om$ assures that it must be $m(\hat t)> 0$, hence $(\hat t,-\bar t)\subseteq A\cup B$ and thus by~\eqref{almostR_1} $\hat t \geq -\bar t - C_7$. Let us now pick some $t^\star \in [\hat t -1, \hat t]$ for which condition~(3) holds, and define $U_1^-:= \Omh(t^\star)$. By definition, $U_1^-$ has unit volume, and \[
\la_1(U_1^-)\leq \la_1(\Om),\qquad {\widetilde {\mathcal F}}_{\alpha,\eps}(U_1^-)< {\widetilde {\mathcal F}}_{\alpha,\eps}(\Om),
\] 
being condition~(3) true for $t^\star$.\par
Observe now that by definition, for every $2\leq p \leq N$, one has $\pi_p\big(\Omt(t^\star)\big)= \pi_p\big(\Omega^+(t^\star)\big)$, hence
\[\begin{split}
{\rm diam} \big(\pi_p(U_1^-)\big) &= {\rm diam} \Big(\pi_p\big(\Omh(t^\star)\big)\Big)
={\rm diam} \Big(\pi_p \Big(\big| \Omt(t^\star) \big|^{-\frac 1N} \Omt(t^\star)\Big)\Big)
\leq 2 \, {\rm diam} \Big(\pi_p\big(\Omt(t^\star)\big)\Big)\\
&= 2 \, {\rm diam} \Big(\pi_p\big(\Omega^+(t^\star)\big)\Big)
\leq 2 \, {\rm diam} \big(\pi_p(\Omega)\big)\,,
\end{split}\]
where we have used that $\big| \Omt(t^\star)\big|\geq 1/2$. 
On the other hand, it is clear from the construction that \[
U_1^-=|\Omt (t^\star)|^{-\frac1N}\Omt(t^\star)\subset 2\Omt(t^\star)\subset \Big\{(x,y)\in \R\times\R^{N-1} : x>2t^\star-2\Big\}.
\]
As a consequence, being $t^\star \geq \widehat t-1\geq -\bar t -C_7-1$, we deduce \[
U_1^-\subset \Big\{(x,y)\in \R\times\R^{N-1} : x>-2\bar t-2C_7-4\Big\}.
\]
Concerning the last part of the claim, recalling again that \[
U_1^-=|\Omt (t^\star)|^{-\frac1N}\Omt(t^\star)\subset 2\Omt(t^\star),
\]
we infer that \[
|U_1^-\setminus [-2\bar t,2\bar t]^N|\leq 2|\Om\setminus [-\bar t,\bar t]^N|\leq \frac{\mh}{2^{2N-1}},
\]
so the proof is concluded also in this case. \qedhere
\end{proof}

In order to conclude our surgery result, we need to iterate~Lemma~\ref{lemmatail}.
First, we apply it to $U_1^-$, in direction $e_1$ for $t\geq 2\bar t=:t_1^+$.
Then we will recursively apply it to the new set that we obtain, in order to get a uniform boundedness in all the other $N-1$ coordinate directions.
We need to take care that, while rescaling, the diameter of the projections in the directions we  already dealt with remains bounded.

For the first step, dealing with $U_1^-$ in direction $e_1$ for $t\geq t_1^+$, Lemma~\ref{threeconditions} can be repeated analogously with a suitable change of notation, for $t\geq t_1^+\geq \bar t$:
\[
\begin{split}
&\Om^+(t)=\Big\{(x,y)\in\Om : x<t\Big\} ,\qquad \Om^-(t)=\Big\{(x,y)\in\Om : x>t\Big\},\\
& \eps(t)=\mathcal H^{N-1}(\Om_t), \qquad m(t)=\int_{t}^{+\infty} \eps(s)\,ds\leq 2\mh,\\
&\phi(t)=\int_{\Om^-(t)}|\nabla u|^2=\int_t^{+\infty}\delta(s)\,ds.
\end{split}
\]
We can then prove the following Lemma.
\begin{lemma}\label{lemmatailright}
Let $\alpha\in(1,N)$.
For every $\eps\leq \overline \eps(N,\alpha)$, given $U_1^-$ the set from Lemma~\ref{lemmatail}, there exists another open, connected set $U_1^+\subseteq\R^N$, still of unit volume, such that 
\begin{enumerate}
\item $\la_1(U_1^+)\leq \la_1(U_1^-)\leq \la_1(\Om)$ and  ${\widetilde {\mathcal F}}_{\alpha,\eps}(U_1^+)\leq{\widetilde {\mathcal F}}_{\alpha,\eps}(U_1^-)\leq {\widetilde {\mathcal F}}_{\alpha,\eps}(\Om)$,
\item $U_1^+\subset \Big\{(x,y)\in \R\times\R^{N-1} : -4C_7-8-4 t_1^+<x<2C_7+4+2t_1^+\Big\}$, and therefore \[
{\rm diam}(\pi_1(U_1^+))\leq C^+_{1}(N,\alpha):=6C_7+12+6t_1^+,
\]
\item $|U_1^+\setminus [-2 t_1^+,2t_1^+]^N|\leq \frac{\mh}{2^{2(N-1)}}\leq \mh$.
\end{enumerate}
\end{lemma}
\begin{proof}
We argue exactly as in the proof of Lemma~\ref{lemmatail}, noting that, since $\la_1(U_1^-)\leq \la_1(\Om)$, the set still satisfies the condition $\la_1(U_1^-)\leq K$ and moreover \[
|U_1^-\cap \{(x,y)\in\R\times\R^{N-1} : x>t_1^+=2\bar t\}|\leq \mh.
\] 
This clearly gives that \[
U_1^+\subset \{(x,y)\in \R\times \R^{N-1} : x<2C_7+4+2t_1^+\}.
\]
Concerning the other bound, it is enough to observe, as in the proof of Lemma~\ref{lemmatail}, that \[
U_1^+\cap \{(x,y) : x<t_1^+\}\subset 2(U_1^-\cap \{(x,y) : x<t_1^+\}).
\]
\end{proof}

We can now iterate the argument in all the other coordinate directions to prove Proposition~\ref{prop:surgery}.
\begin{proof}[Proof of Proposition~\ref{prop:surgery}]
First of all, we note that $K(N,\alpha)\geq \la_1(B)+\overline \delta$ and we set $\overline \delta$ as in~\eqref{eq:asimmpiccolamhat}, $\overline \eps(N,\alpha)$ as in Lemma~\ref{threeconditions}. 
Let us fix $\Om\subset \R^N$ an open (smooth) set of unit volume with $\la_1(\Om)\leq K$.
Thanks to Lemma~\ref{lemmatail} and Lemma~\ref{lemmatailright}, we have found a new open and connected set of unit measure $U_1^+$ with \[
\la_1(U_1^+)\leq \la_1(\Om),\qquad {\widetilde {\mathcal F}}_{\alpha,\eps}(U_1^+)\leq {\widetilde {\mathcal F}}_{\alpha,\eps}(\Om),
\]
and moreover \[
{\rm diam}(\pi_1(U_1^+))\leq C_1^+(N,\alpha),\qquad |U_1^+\setminus [-2t_1^+,2t_1^+]^N|\leq \frac{\mh}{2^{2(N-1)}}.
\]
We now iterate the argument in all the remaining directions.
Let us show it for $e_2$.
We call first $t_2^-:=2t_1^+=4\bar t$ and repeat Lemma~\ref{lemmatail} to $U_1^+$ in direction $e_2$ for $z_2\leq -t_2^-$ so that we can find a new open and connected set of unit measure $U_2^-$ such that \[
\begin{split}
&\la_1(U_2^-)\leq \la_1(U_1^+),\qquad {\widetilde {\mathcal F}}_{\alpha,\eps}(U_2^-)\leq {\widetilde {\mathcal F}}_{\alpha,\eps}(U_1^+),\\
&{\rm diam}(\pi_1(U_2^-))\leq 2 C_1^+(N,\alpha),\qquad U_2^-\subset \{z\in\R^N : z_2>-2C_7-4-2t_2^-\},\\
&|U_2^-\setminus [-2t_2^-,2t_2^-]^N|\leq \frac{\mh}{2^{2(N-1)-1}}.
\end{split}
\]
Then we call $t_2^+=2t_2^-$ and repeat Lemma~\ref{lemmatailright} to $U_2^-$ in direction $e_2$ for $z_2\geq t_2^+$ in order to find a new open and connected set of unit measure $U_2^+$ such that  
\[
\begin{split}
&\la_1(U_2^+)\leq \la_1(U_2^-),\qquad {\widetilde {\mathcal F}}_{\alpha,\eps}(U_2^+)\leq {\widetilde {\mathcal F}}_{\alpha,\eps}(U_2^-),\\
&{\rm diam}(\pi_1(U_2^+))\leq 2^2 C_1^+(N,\alpha),\qquad {\rm diam}(\pi_2(U_2^+))\leq C_2^+(N,\alpha),\\
&|U_2^+\setminus [-2t_2^-,2t_2^-]^N|\leq \frac{\mh}{2^{2(N-2)}},
\end{split}
\]
where we can take $C_2^+(N,\alpha):=6C_7+12+6t_2^+$.
We note that the last condition \[
|U_2^-\setminus [-2t_2^+,2t_2^+]^N|\leq \frac{\mh}{2^{2(N-2)}}\leq \mh,
\]
is needed so that we can restart the cutting procedure in direction $e_3$ knowing that $|U_2^+\cap \{z_3\leq -2t_2^+\}|\leq \mh$, which is the condition required for Lemma~\ref{primastima}.
Iterating this procedure other $N-2$ times, we obtain in the end an open and connected set of unit measure $U_N^+$ such that \[
\begin{split}
&\la_1(U_N^+)\leq \la_1(\Om),\qquad {\widetilde {\mathcal F}}_{\alpha,\eps}(U_N^+)\leq {\widetilde {\mathcal F}}_{\alpha,\eps}(\Om),\\
&{\rm diam}(\pi_p(U_N^+))\leq 2^{2(N-p)} C_p^+(N,\alpha),\qquad \text{for }p=1,\dots, N,
\end{split}
\]
where $C_p^+(N,\alpha):=6C_7+12+6t_p^+$ and $t_p^+=2^{2p-1}\bar t$ for all $p=1,\dots N$.
Clearly $\Omh=U_N^+$ is a good choice for proving the claim.
\end{proof}

\section{Proof of Theorem~\ref{thm:mainvero}}\label{maintheorem}

As outlined in the introduction, the proof of Theorem \ref{thm:mainvero} can be obtained as the juxtaposition of two independent results. Hence, we divide the section in two parts. In the first one, we prove the minimality of the ball for the functional ${\widetilde {\mathcal F}}_{\alpha,\eps}$ among equibounded sets. Then we use the surgery argument of Section~\ref{sect:surgery} to conclude the proof of Theorem~\ref{thm:mainvero}.
\subsection{Rigidity of the ball in the class of equibounded sets}  
We aim to prove the following result.
\begin{proposition}\label{thm:minla1BR}
Let $\alpha\in(1,N)$ and $R>R_0$. There is $\eps^R_{\la_1}=\eps^R_{\la_1}(N,\alpha,R)$ such that, for all $\eps\leq \eps^R_{\la_1}$, the unit ball is the unique minimizer for problem
\begin{equation}\label{eq:minla1BR}
\min\Big\{\la_1(\Om)+\eps V_\alpha(\Om) : \Om\subset B_R,\;|\Om|=1\Big\}.
\end{equation}
\end{proposition}

The previous result is  an easy consequence of Theorem \ref{thm:mainvero2} and the Kohler-Jobin inequality \cite{kj,brakj}, which we recall.
\begin{lemma}[Kohler-Jobin inequality]
For every open set $\Om\subset \R^N$ of finite measure, it holds the following scale invariant inequality
\begin{equation}\label{eq:KJ}
\frac{\la_1(\Om)}{\la_1(B)}\geq \left(\frac{(-E(B))}{(-E(\Om))}\right)^{\frac{2}{N+2}},
\end{equation}
where $B\subset\R^N$ is any ball  in $\R^N$.
Equality holds if and only if $\Om$ is   a ball up to sets of null capacity.
\end{lemma}

\begin{proof}[Proof of Proposition \ref{thm:minla1BR}]
We follow here a smart and easy technique proposed in \cite{brdeve}.
Let $\Om\subset B_R$ be an (open) set with $|\Om|=1$ and we call $B$ a ball of unit measure centered for the sake of simplicity in the origin, as $B_R$.
By rewriting \eqref{eq:KJ}  as 
\begin{equation}\label{eq:KJ1}
\frac{\la_1(\Om)}{\la_1(B)}-1\geq \left(\frac{(-E(B))}{(-E(\Om))}\right)^\vartheta-1,\qquad \text{with }\vartheta=\frac{2}{N+2},
\end{equation}
and since $\vartheta\in(0,1)$, by concavity we have \[
t^\vartheta-1\geq (2^\vartheta-1)(t-1),\qquad \text{for all }t\in[1,2].
\]
If $E(B)\geq 2E(\Om)$, then we have, using~\eqref{eq:KJ1},\[
\frac{\la_1(\Om)}{\la_1(B)}-1\geq c_\vartheta\left(\frac{(-E(B))}{(-E(\Om))}-1\right),
\]
with $c_\vartheta=2^\vartheta-1$. The above inequality implies that \[
\la_1(\Om)-\la_1(B)\geq c_\vartheta\la_1(B)\frac{(E(\Om)-E(B))}{(-E(\Om))}\geq \frac{c_\vartheta\la_1(B)}{(-E(B))}(E(\Om)-E(B)),
\]
where we have used also the Saint-Venant inequality.

Since, by Theorem~\ref{thm:mainvero2} for all $\eps\leq \eps_{E}(N,\alpha,R)$, the ball of unit measure is the only minimizer for the functional \[
\Om\mapsto E(\Om)+\eps V_\alpha(\Om),\qquad \text{for } \Om\subset B_R,\;|\Om|=1,
\] 
we deduce that, \[
\la_1(\Om)-\la_1(B)\geq \frac{c_\vartheta\la_1(B)}{(-E(B))}(E(\Om)-E(B))\geq \eps \frac{c_\vartheta\la_1(B)}{(-E(B))}( V_\alpha(B)- V_\alpha(\Om)).
\]

On the other hand, if $E(B)<2E(\Om)$, we can still obtain from~\eqref{eq:KJ1}\[
\frac{\la_1(\Om)}{\la_1(B)}-1\geq 2^\vartheta-1\geq \frac{2^\vartheta-1}{(-E(B))}(E(\Om)-E(B))\geq \frac{2^\vartheta-1}{(-E(B))}\eps( V_\alpha(B)- V_\alpha(\Om)).
\]

In conclusion, we have proved that, for all \[
\eps\leq \eps^R_{\la_1}=\eps_{E} \frac{c_\vartheta\la_1(B)}{(-E(B))},
\]
we have that \[
\la_1(\Om)+\eps V_\alpha(\Om)\geq \la_1(B)+\eps V_\alpha(B),\qquad \text{for all }\Om\subset B_R,\;|\Om|=1,
\]
and the proof is concluded.
\end{proof}

\subsection{Conclusion of the proof of Theorem~\ref{thm:mainvero}}
We are now in position to prove the main result of the paper.
\begin{proof}[Proof of Theorem~\ref{thm:mainvero}]
First of all, as it is standard when dealing with quantitative inequalities, we can perform the minimization only in the class of sets with $\la_1(A)-\la_1(B)\leq \overline \delta(N,\alpha)$.
In fact, it is clearly enough to take $\eps_{\la_1}< \overline \delta / V_\alpha(B)$: in the other case when $\la_1(A)-\la_1(B)>\overline \delta$, and $\eps\leq \eps_{\la_1}$ we have \[
\la_1(A)+\eps V_\alpha(A)\geq \la_1(B)+\overline \delta+\eps V_\alpha(A)\geq \la_1(B)+\eps V_\alpha(B),
\]
thus the claim holds for free.

Clearly one can take a minimizing sequence $(A_n)_{n\in\N}$ made of smooth sets.
%In fact, first of all, we call $\Om_R$ an optimal set for minimization problem in the ball $B_R$ for all integer $R> R_0$ \[
%\min{\left\{{\widetilde {\mathcal F}}_{\alpha,\eps}(A) : A\subset B_R,\;|A|=1\right\}}.
%\]
%The existence of a minimizer follows from Theorem~\ref{thm:minla1BR} and Theorem~\ref{thm:noconstraint}.
%Moreover, thanks to Theorem~\ref{thm:bdvthm4.18}, we know that the sets $\Om_R$ are at least of class $C^{1,\gamma}$, though with constants that can explode as $R\to +\infty$.
%It is clear that $(\Om_R)_{R\geq R_0}$ is a minimizing sequence made of smooth sets.
%
We show that the elements of the minimizing sequence $(A_n)$ can be chosen to be also connected.
Let us take a smooth open set of unit measure $A$, which is made of an at most countable number of connected components, \[
A=\cup_{k\in\N}A^k.
\]
For all $\vartheta>0$, we consider the segment $S_{k}$ connecting the components $A^k$ and $A^{k+1}$ and we consider $T_{k,\vartheta}=\cup_{x\in S_k}B_{\zeta}(x)$, choosing $\zeta$ so that $|T_{k,\vartheta}|\leq \frac{\vartheta}{2^k}$.
We call now \[
A_\vartheta:=\cup_{k\in \N}T_{k,\vartheta}\cup A\supset A,\qquad {\widehat A}_{\vartheta}:=|A_\vartheta|^{-1/N}A_\vartheta,
\]
and note that \[
 |A_\vartheta|\leq  1+\vartheta,\qquad |{\widehat A}_\vartheta|=1.
\]
By monotonicity and scaling of the eigenvalue, we immediately deduce \[
\la_1(A_\vartheta)\leq \la_1(A),\qquad \la_1({\widehat A}_{\vartheta})\leq |A_\vartheta|^{2/N}\la_1(A)\leq (1+\frac{2}{N}\vartheta)\la_1(A).
\]
On the other hand, by Lemma~\ref{le:knupfermuratov} and scaling of the Riesz energy, \[
V_\alpha({\widehat A}_\vartheta)=|A_\vartheta|^{-\frac{N+\alpha}{N}}V_\alpha(A_\vartheta)\leq V_\alpha(A)+C(N,\alpha)\vartheta.
\]

All in all, we have that, for all $\vartheta>0$, \[
{\widetilde {\mathcal F}}_{\alpha,\eps}(\widehat A_\vartheta)\leq {\widetilde {\mathcal F}}_{\alpha,\eps}(A)+C(N,\alpha)\vartheta,
\]
and by arbitrariety of $\vartheta$, applying this procedure to all the elements of the minimizing sequence, we deduce that assuming all elements of the sequence to be connected is not restrictive.

At this point we can take $(A_n)_{n\in\N}$ a minimizing sequence for problem~\eqref{eq:minla1}, made of smooth, connected sets of unit measure.
Up to take $\eps_{\la_1}\leq \overline \eps$, for all $n$, we can apply Proposition~\ref{prop:surgery} to $A_n$ and we find a new open and connected set, $\widehat A_n$, of unit measure and with \[
{\widetilde {\mathcal F}}_{\alpha,\eps}(\widehat A_n)\leq {\widetilde {\mathcal F}}_{\alpha,\eps}(A_n),\qquad {\rm diam}(\widehat A_n)\leq D(N,\alpha).
\] 
Hence, $(\widehat A_n)_{n\in \N}$ is still a minimizing sequence for problem~\eqref{eq:minla1BR}, made by sets with uniformly bounded diameter.
It is eventually enough to restrict the minimization problem to a ball $B_R$ with $R=D(N,\alpha)$, and we can find that the unit ball is the unique optimal set for problem~\eqref{eq:minla1BR}, thanks to Proposition~\ref{thm:minla1BR}. We note that, since now $R$ has been fixed equal to $D(N,\alpha)$, then the constant $\eps^R_{\la_1}$ now depends only on $N,\alpha$ and we conclude taking $\eps_{\la_1}=\min\{\eps^{D(N,\alpha)}_{\la_1},\overline \eps, \overline \delta/V_\alpha(B)\}$.
\end{proof}

\section{The non-existence results}\label{nonex}
In this section we show Theorem \ref{thm:nonexistballcond}.

\begin{proof}[Proof of Theorem \ref{thm:nonexistballcond}]
We give the proof just for the case of the torsion energy, since the other one is analogous.
Notice that any set in $\mathcal U(\delta)$ is bounded. Moreover any minimizer  must be connected. Otherwise, if $\Omega$ is made up by the union of two disjoint open sets $\Om_1$ and $\Om_2$ we have that $E(\Omega)$ does not change by sending toward infinity $\Om_1$ while keeping $\Om_2$ fixed, while $V_\alpha$ under such a translation strictly decreases, contradicting the minimality. 
Now, any connected open set lying in  $\mathcal U(\delta)$ has bounded diameter, with a bound depending only on $\delta$ and $N$:
\[
{\rm diam}(\Om)\leq d(\delta),\qquad \text{for all }\Om\in \mathcal U(\delta).
\]
{
By choosing an horizontal open necklace-type set of $k=k_\delta =\lfloor \delta^{-N}\rfloor$\footnote{For any $x\in \R$, we denote $\lfloor x\rfloor$ the integer part of $x$.} tangent balls of radius $c(N,\delta)\delta$ disposed on a line, one finds that  there is a geometric constant $C(N)$ such that $d(\delta)\le C(N)\delta^{1-N}$. Here $c(N,\delta)=\delta^{-N}/\lfloor \delta^{-N}\rfloor\in(1,2)$ is so that the necklace-type set has total measure $1$.
Therefore, for all $\Om\in \mathcal U(\delta)$, we can compute the following lower bound on $\mathcal F_{\alpha,\eps}$:
\begin{equation}\label{eq:lowbound}
\mathcal F_{\alpha,\eps}(\Om)=E(\Om)+\eps V_\alpha(\Om)\geq E(B)+\eps\frac{1}{d(\delta)^{N-\alpha}}\ge E(B)+\eps C(N) \delta^{(N-1)(N-\alpha)},
\end{equation}
where we have used the Saint-Venant inequality.

We now construct a disconnected set with energy lower than the right-hand side of \eqref{eq:lowbound}, as long as $N-\alpha<1$. This will immediately  entail non-existence of minimizers.
Let $k=k_\delta\in\mathbb N$ and $d(\delta)$ be the  parameters defined above. We define the set $\widetilde \Om:=\bigcup_{i=1}^k B_r(x_i)$, with $|B_r(x_i)|=\frac1k$ for all $i\in\{1,\dots,k\}$ and where $x_i \in\R^n$ are chosen so that the balls are mutually disjoint. By construction $\widetilde\Omega\in\mathcal U(\delta)$.
We set 
\[
q=\min \Big\{ |x_{i+1}-x_i|  : i=1,\dots, N-1 \Big\}.
\]
Notice that we can choose $q$ as large as we want, up to change the values of the points $x_i$ (which can be arbitrarily mutually distant). 
Letting $B$ be the ball of unit measure, by scaling we have 
 \begin{equation}\label{eq:upbound}
 \begin{aligned}
\mathcal F_{\alpha,\eps}(\widetilde \Om)&=  k^{-\frac{2}{N}}E(B)+\eps k^{-\frac{\alpha}{N}}V_\alpha(B)+\eps\sum_{i\not= j}\int_{B_r(x_i)}\int_{B_r(x_j)}\frac{1}{|x-y|^{N-\alpha}}\,dxdy \\
&\le c_0(N)\left(  k^{-\frac{2}{N}}E(B)+\eps k^{-\frac{\alpha}{N}}V_\alpha(B)+\frac{k(k-1)\eps}{k^2 q^{N-\alpha}}\right)\\
&\le c_0(N)\left( k^{-\frac{2}{N}}E(B)+\eps k^{-\frac{\alpha}{N}}V_\alpha(B)+\frac{ \eps}{  q^{N-\alpha}}\right)\\
&\leq c_1(N) \left( E(B)\delta^2 +\eps V_\alpha(B)\delta^\alpha+ \eps q^{\alpha-N}    \right) \\
&\leq 2c_1(N) \left( E(B)\delta^2 +\eps V_\alpha(B)\delta^\alpha   \right),
\end{aligned}
\end{equation}
where $c_0(N)$, $c_1(N)$ are geometric constants and the last inequality holds by choosing $q$ big enough. Hence, by combining \eqref{eq:lowbound} and \eqref{eq:upbound}, setting $c(N):=2c_1(N)$ we contradict the minimality as soon as 
\[
E(B)+\eps  \delta^{(N-1)(N-\alpha)}>c(N) \left( E(B)\delta^2 +\eps V_\alpha(B)\delta^\alpha   \right),
\]
that is  if 
\[
(-(E(B))\left(1-c(N)\delta^2 \right)\le \eps \left(C(N)\delta^{(N-1)(N-\alpha)}-c(N)V_\alpha(B)\delta^\alpha \right):=\eps u_{\alpha,N}(\delta)
\]
Since $\alpha>N-1$, then $(N-1)(N-\alpha)<\alpha$ so that  there exists $1>\delta_0(N,\alpha)>0$ such that for all $0<\delta<\delta_0$ it holds
\[
u_{\alpha,N}(\delta) =C(N)\delta^{(N-1)(N-\alpha)}-c(N)V_\alpha(B)\delta^\alpha >0,\qquad\text{and}\qquad1-c(N)\delta^2\ge\frac12.
\]
Thus the contradiction follows as long as 
\[
\eps\ge\eps_{max}(\alpha,N,\delta):=\frac{(-(E(B))\left(1-c(N)\delta^2 \right)}{2u_{\alpha,N}(\delta)}.
\]
This concludes the proof.
}
\end{proof}
\

\

\

%\newpage


\begin{thebibliography}{999}

\bibitem{aco}
G.~Alberti, R.~Choksi, and F.~Otto.
 Uniform energy distribution for an isoperimetric problem with
  long-range interactions.
  J. Amer. Math. Soc. {\bf 22} (2009), 569--605.

\bibitem{afp}
L.~Ambrosio, N.~Fusco, and D.~Pallara.
\newblock {\em Functions of bounded variation and free discontinuity problems}.
\newblock Oxford Mathematical Monographs. Oxford University Press, New York,
  2000.

\bibitem{agalca} N. Aguilera, H.W. Alt, L.A. Caffarelli, An optimization problem with volume constraint. SIAM J.
Control Optim., {\bf24} (1986), 191--198.

\bibitem{alca} H.W.~Alt, L.A.~Caffarelli, Existence and regularity for a minimum problem with free boundary, J. Reine Angew. Math. {\bf325} (1981), 105--144.	

\bibitem{amti} L.~Ambrosio, P.~Tilli, Topics on analysis in metric spaces, Oxford Lecture Series in Mathematics and its Applications, 25. Oxford University Press, Oxford, 2004. viii+133 pp.

\bibitem{bonacini14}
M.~Bonacini and R.~Cristoferi.
 Local and global minimality results for a nonlocal isoperimetric
  problem on {$\mathbb R^N$}.
  SIAM J. Math. Anal. {\bf46} (2014), 2310--2349.

\bibitem{brakj} L.~Brasco,  On torsional rigidity and principal frequencies: an invitation to the Kohler-Jobin rearrangement technique, ESAIM Control Optim. Calc. Var. {\bf20} (2014), no. 2, 315--338.

\bibitem{brdeve} L.~Brasco, G.~De~Philippis, B.~Velichkov, Faber-Krahn inequalities in sharp quantitative form. 
Duke Math. J. {\bf164} (2015), no. 9, 1777--1831.

%\bibitem{bhp} T.~Brian\c{c}on, M.~Hayouni, M.~Pierre, Lipschitz continuity of state functions in some optimal shaping, Calc. Var. PDE {\bf23} (1) (2005) 13--32.

\bibitem{bu} D.~Bucur, Minimization of the k-th eigenvalue of the Dirichlet Laplacian, Arch. Rational Mech. Anal. {\bf 206} (3) (2012), 1073--1083.
 

%\bibitem{bubu} D.~Bucur, G.~Buttazzo, Variational methods in shape optimization problems. Progress in Nonlinear Differential Equations and their Applications,65. Birkh\"auser Boston, Inc., Boston, MA, 2005. viii+216 pp. 

%\bibitem{bubuhe} D.~Bucur, G.~Buttazzo, A.~Henrot, Minimization of $\lambda_2(\Omega)$ with a perimeter constraint. Indiana Univ. Math. J. 58 (2009), no. 6, 2709--2728.

\bibitem{buma} D.~Bucur, D.~Mazzoleni, A surgery result for the spectrum of the Dirichlet Laplacian, SIAM J. Math. Anal. {\bf47} (2015), no. 6, 4451--4466.

\bibitem{bmpv} D.~Bucur, D.~Mazzoleni, A.~Pratelli, B.~Velichkov, Lipschitz regularity of the eigenfunctions on optimal domains, Arch. Ration. Mech. Anal. {\bf 216} (1) 117--151 (2015).

\bibitem{cp1}
R.~Choksi and M.~A. Peletier.
\newblock Small volume fraction limit of the diblock copolymer problem: {I.
  Sharp} interface functional.
\newblock { SIAM J. Math. Anal.} {\bf42} (2010), 1334--1370.

\bibitem{cp2}
R.~Choksi and M.~A. Peletier.
\newblock Small volume fraction limit of the diblock copolymer problem: {II.
  Diffuse} interface functional.
\newblock {SIAM J. Math. Anal.} {\bf43} (2011), 739--763.


\bibitem{demamu} G. De Philippis, M. Marini, E. Mukoseeva, The sharp quantitative isocapacitary inequality, Rev. Mat. Iberoam. Electronically published on February 1, 2021. doi: 10.4171/rmi/1259 (to appear in print).

\bibitem{f2m3}
A.~Figalli, N.~Fusco, F.~Maggi, V.~Millot, and M.~Morini.
\newblock Isoperimetry and stability properties of balls with respect to
  nonlocal energies.
\newblock {Commun. Math. Phys.} {\bf336} (2015), 441--507.

\bibitem{FigMag}
A. Figalli, F. Maggi.
\newblock On the shape of liquid drops and crystals in the small mass regime. 
\newblock { Arch. Ration. Mech. Anal.} {\bf 201} (1) (2011), 143--207.

\bibitem{frank15}
R.~L. Frank and E.~H. Lieb.
\newblock A compactness lemma and its application to the existence of
  minimizers for the liquid drop model.
\newblock { SIAM J. Math. Anal.} {\bf47} (2015), 4436--4450.

\bibitem{fuprariesz} N.~Fusco, A.~Pratelli, Sharp stability for the Riesz potential, ESAIM: COCV
Volume 26, (2020) article number 113.

\bibitem{gamow30}
G.~Gamow.
\newblock Mass defect curve and nuclear constitution.
\newblock {Proc. Roy. Soc. London A} {\bf126} (1930), 632--644.

\bibitem{gnr1}
M.~Goldman, M.~Novaga, and B.~Ruffini.
\newblock Existence and stability for a non-local isoperimetric model of
  charged liquid drops.
\newblock {Arch. Ration. Mech. Anal.}  {\bf217} (2015),1--36.

\bibitem{gnr2}
M.~Goldman, M.~Novaga, and B.~Ruffini.
\newblock On minimizers of an isoperimetric problem with long-range
interactions and convexity constraint. 
\newblock  { Anal. PDE} {\bf11} (2018), no. 5, 1113–1142.


\bibitem{gush}B.~Gustafsson, H.~Shahgholian, Existence and geometric properties of solutions of a free boundary problem in potential theory,
J. Reine Angew. Math. {\bf473} (1996), 137--179.



\bibitem{hpnew} A.~Henrot, M.~Pierre, Shape variation and optimization.
A geometrical analysis. EMS Tracts in Mathematics, {\bf28}. European Mathematical Society (EMS), Z\"urich, 2018. xi+365 pp.


\bibitem{julin14}
V.~Julin.
\newblock Isoperimetric problem with a {Coulombic} repulsive term.
\newblock {Indiana Univ. Math. J.} {\bf63} (2014), 77--89.

\bibitem{kini} D.~Kinderlehrer, L.~Nirenberg, Regularity in free boundary problems, Ann. Scuola Norm. Sup. Pisa Cl. Sci.  {\bf4} (1977), no. 2, 373--391.
	
\bibitem{km1}{H. Kn\"upfer and C. B. Muratov}.
\newblock  On an isoperimetric problem with a competing non-local
  term I. The planar case. 
\newblock { Comm. Pure Appl. Math.} {\bf66} (2013), 1129--1162.

	
\bibitem{km}  H.~Kn\"upfer, C.~B.~Muratov, On an isoperimetric problem with a competing nonlocal term II: The general case. Comm. Pure Appl. Math. {\bf67} (2014), no. 12, 1974--1994.

\bibitem{kj} M.-T. Kohler-Jobin, Une m\'ethode de comparaison isop\'erim\'etrique de fonctionnelles de domaines de
la physique math\'ematique. I. Une d\'emonstration de la conjecture isop\'erim\'etrique $P\la^2\geq \pi j_0^4/2$ de
P\'olya et Szeg\"o, Z. Angew. Math. Phys., {\bf29} (1978), 757--766.

\bibitem{LL} E.~H.~Lieb, M.~Loss, Analysis, Graduate Studies in Mathematics, 14. American Mathematical Society, Providence, RI, 1997. xviii+278 pp.

\bibitem{lu14}
J.~Lu and F.~Otto.
\newblock Nonexistence of minimizer for {Thomas-Fermi-Dirac-von
  Weizs\"acker} model.
\newblock {Comm. Pure Appl. Math.} {\bf67} (2014), 1605--1617.

\bibitem{maggi} F.~Maggi, Sets of finite perimeter and geometric variational problems. An introduction to geometric measure theory, Cambridge Studies in Advanced Mathematics, 135. Cambridge University Press, Cambridge, 2012. xx+454 pp.

\bibitem{mp} D.~Mazzoleni, A.~Pratelli, Existence of minimizers for spectral problems. J. Math. Pures Appl. (9) 100 (2013), no. 3, 433--453.

\bibitem{mtv} D.~Mazzoleni, S.~Terracini, B.~Velichkov, Regularity of the optimal sets for some spectral functionals, Geom. Funct. Anal. {\bf27} (2) (2017), 373--426.


\bibitem{MurNovRuf}
C. Muratov, M. Novaga and B. Ruffini.
\newblock On equilibrium shapes of charged flat drops.
\newblock {Comm. Pure Appl. Math.} {\bf71} (6) (2018), 1049--1073.


\bibitem{PS} G. P\'olya, G. Szeg\"{o}, {\it Isoperimetric Inequalities in Mathematical Physics}. Annals of Mathematics Studies, no. 27, Princeton University Press, Princeton, N. J., 1951

%\bibitem{rose} X.~Ros-Oton,  J.~Serra, The Dirichlet problem for the fractional Laplacian: regularity up to the boundary, J. Math. Pures Appl. (9) {\bf101} (2014), no. 3, 275--302.

\bibitem{velectures} B.~Velichkov, Regularity of the one-phase free boundaries, Lecture notes available at {\tt http://cvgmt.sns.it/paper/4367/}

	
\end{thebibliography}
\end{document}